\documentclass[a4paper,11pt]{article}
\usepackage{amsthm,amsmath}
\usepackage{amssymb}
\usepackage[latin1]{inputenc}
\usepackage[all]{xy}%{xypic}
\usepackage{enumerate}
\usepackage[T1]{fontenc}
\usepackage{a4wide}
\usepackage[mathscr]{eucal}
\usepackage{graphicx}
\usepackage{color}

\newtheorem{theorem}{Theorem}[section]
\newtheorem{proposition}[theorem]{Proposition}
\newtheorem{corollary}[theorem]{Corollary}
\newtheorem{lemma}[theorem]{Lemma}
\newtheorem*{theorem*}{Theorem}
\newtheorem*{proposition*}{Proposition}
\newtheorem*{corollary*}{Corollary}
\newtheorem*{lemma*}{Lemma}
\theoremstyle{definition}
\newtheorem{definition}[theorem]{Definition}

\newtheorem{example}[theorem]{Example}
\newtheorem{remark}[theorem]{Remark}
\newtheorem*{remark*}{Remark}

\newtheorem*{definition*}{Definition}

\newcommand{\tensor}[1]{\otimes_{#1}}

\renewcommand{\hom}[3]{\mathrm{Hom}_{#1}(#2,\,#3)}

\newcommand{\Endo}[2]{\mathrm{End}{}_{#1}(#2)}

\newcommand{\bara}[1]{\overline{#1}}
\newcommand{\Aut}[2]{\mathbf{Aut}_{#1}(#2)}
\newcommand{\aut}[1]{\mathbf{Aut}(#1)}

\newcommand{\fk}[1]{\mathfrak{#1}}
\newcommand{\Sf}[1]{\mathsf{#1}}

\newcommand{\lr}[1]{\left(\underset{}{} #1 \right)}

\newcommand{\Scr}[1]{\mathscr{#1}}
\newcommand{\Inv}[2]{\mathbf{Inv}_{#1}(#2)}
\newcommand{\Picar}[1]{\mathbf{Pic}(#1)}
\newcommand{\RPicar}[2]{\mathbf{Pic}_{#1}(#2)}

\newcommand{\objeto}[3]{\xymatrix@C=35pt{ #1
\ar@2{->}|-{\,[#2]\,}[r] & #3}}
\newcommand{\Z}{\mathcal{Z}}
\newcommand{\G}{\mathcal{G}}
\newcommand{\F}{\mathcal{F}}

\newcommand{\U}{\mathcal{U}}

\begin{document}
\allowdisplaybreaks

%Title and Authors
\title{Invertible unital bimodules over rings with local units, and
related exact sequences of groups, II
\footnote{Research supported by the grant MTM2010-20940-C02-01 from the
Ministerio de Ciencia e Innovaci\'on and from FEDER}}
\author{L. El Kaoutit \\
\normalsize Departamento de \'{A}lgebra \\ \normalsize Facultad de
Educaci\'{o}n y Humanidades \\
\normalsize Universidad de Granada \\ \normalsize El Greco N${}^{0}$
10, E-51002 Ceuta, Espa\~{n}a \\ \normalsize
e-mail:\textsf{kaoutit@ugr.es} \and J. G\'omez-Torrecillas \\
\normalsize Departamento de \'{A}lgebra \\ \normalsize Facultad de
Ciencias \\
\normalsize Universidad
de Granada\\ \normalsize E18071 Granada, Espa\~{n}a \\
\normalsize e-mail: \textsf{gomezj@ugr.es} }

\date{}

\maketitle

\begin{abstract}
Let $R$ be a ring with a  set of local
units, 
and a homomorphism of groups $\underline{\Theta} : \G \to \Picar{R}$ to the
Picard group of $R$. We study under which conditions  $\underline{\Theta}$ is determined by a factor map, and, henceforth, it defines a generalized crossed product with a same set of local units.  Given a ring extension $R \subseteq S$ with the same set of local units and assuming that $\underline{\Theta}$ is induced by a homomorphism of groups $\G \to \Inv{R}{S}$ to the group of all invertible $R$-sub-bimodules of $S$, then we construct an
analogue of the Chase-Harrison-Rosenberg
seven terms exact sequence of groups attached to the triple $(R
\subseteq S, \underline{\Theta})$, which involves the first,
the second and the third cohomology groups of $\G$ with coefficients
in the group of all $R$-bilinear
automorphisms of $R$. Our approach generalizes the works by Kanzaki and Miyashita in the unital case. 
\end{abstract}

\section*{Introduction}

For a  Galois extension $R/\Bbbk$ of commutative rings with identity and  finite Galois group $G$,
Chase, Harrison and Rosenberg,
gave in \cite[Corollary 5.5]{Chase et al:1965} (see also
\cite{Chase/Rosenberg:1965})
the following  seven terms exact sequence of groups
\begin{multline*}
1 \to H^1(G,\U(R)) \to \RPicar{\Bbbk}{\Bbbk} \to \RPicar{R}{R}^G \to
H^2(G, \U(R)) \to B(R/\Bbbk) \\ \to H^1(G, \RPicar{R}{R}) \to
H^3(G,\U(R)),
\end{multline*}
where for a ring $A$ with identity, $\RPicar{A}{A}\,=\,\{ [P] \in
\Picar{A}|\,\,
ap=pa, \forall \, a \in A, p \in P\}$ is a subgroup of the Picard
group $\Picar{A}$, and $\U(A)$ the group of units.
The group $B(R/\Bbbk)$ is the Brauer group of Azumaya
$\Bbbk$-algebras split by $R$,
and the other terms are the usual cohomology groups of $G$ with
coefficients in abelian groups.
As one can realize, the case of finite Galois extension
of fields
reduces to fundamental theorems in Galois cohomology of fields extensions, 
namely, the
Hilbert's 90 Theorem, and the isomorphism of the Brauer group of
a field with the second cohomology group.

The construction of this sequence of groups, as was given in
\cite{Chase et al:1965,Chase/Rosenberg:1965},
is ultimately based on the consideration of a certain spectral 
sequence already introduced by Grothendieck, and the generalized
Amitsur cohomology.
The same sequence was also constructed after that by Kanzaki
\cite[Theorem, p.187]{Kanzaki:1968},
using only elementary methods that employ the novel notion of
generalized crossed products.
This notion has been the key success which allowed  
Miyashita to extend in \cite[Theorem 2.12]{Miyashita:1973} the above
framework to the context of a noncommutative unital
ring extension $R\subseteq S$ with a homomorphism of groups to the
group of
all invertible sub-$R$-bimodules of $S$. In this setting the seven
terms exact
sequence takes of course a slightly different form,
although the first, the fourth and the last terms remain the
cohomology groups of the acting group with
coefficients in the group of units of the base ring.
The commutative Galois case is then recovered by first considering a
tower $\Bbbk \subseteq R \subseteq S$, where
$\Bbbk \subseteq R$ is a finite Galois extension and $S$ is a fixed
crossed product constructed
from the Galois group; secondly by considering the homomorphism of
groups which sends
any element in the Galois group to it corresponding homogeneous
component of $S$
(components which belong to the group of all invertible
sub-$R$-bimodules of $S$).

In this paper we construct an analogue of the above seven terms
exact sequence in the context
of rings with local units (see the definition below). The most
common class of examples are the rings
with enough orthogonal idempotents or Gabriel's rings which are
defined from small additive categories.
It is noteworthy that up to our knowledge
there is no satisfactory notion in the literature of morphism of
rings with local units that
could be, for instance,  extracted from a given additive functor
between additive small categories.
Here we restrict ourselves to the naive case of extensions of the
form $R \subseteq S$ which
have the same set of local units. This at least includes the
situation where two small additive categories
have the same set of objects and differ only in the size of the sets
of morphisms, as it frequently happens
in the theory of localization in abelian categories. Our methods are
inspired from Miyashita's
work \cite{Miyashita:1973} and use the results of our earlier work
\cite{ElKaoutit/Gomez:2010}. The present
paper is in fact a continuation of \cite{ElKaoutit/Gomez:2010}. It
is worth noticing that,
although the steps for constructing the seven terms sequence in our
context are slightly parallel to \cite{Miyashita:1973}, 
this construction is not an easy task since it has its own
challenges and difficulties.

The paper is organized as follows. In Section \ref{sec:1} we give
some preliminary results on similar and
invertible unital bimodules. Most of results on similar unital
bimodules are in fact the non unital version
of \cite{Hirata:1968, Hirata:1969}, except perhaps the construction
of the twisting natural transformations
and its weak associativity (Proposition \ref{prop:1} and
Lemma \ref{lema:3}).
The generalized crossed product is introduced in Section
\ref{sec:2}, where we also give useful
constructions on the normalized $2$ and $3$-cocycles with
coefficient in the unit group of the base ring. 
Section \ref{sec:3} is devoted to the construction of an abelian
group whose elements consist of
isomorphic classes of generalized crossed products. We show
that it contains a subgroup which is
isomorphic to the $2$-cohomology group (Proposition \ref{prop:5}).
Our main results is contained in Section
\ref{sec:4}, which contains a number of exact sequences of groups that culminate in the seven terms exact sequence that generalizes Chase-Harrison-Rosenberg's one (Theorem \ref{them} ).

\medskip
{\textbf{Notations and basic notions}:}
\smallskip

In this paper ring means an associative and not necessarily unital
ring. We denote by $\Scr{Z}(R)$ (resp. $\Scr{Z}(G)$)
the center of a ring $R$ (resp. of a group $G$), and if $R$ is
unital we will denote by $\U(R)$ the unit group of $R$,
that is, the set of all invertible elements of $R$.

\medskip

Let $R$ be a ring and $\Sf{E} \subset R$ a set of idempotent elements;  $R$ is said to be a \emph{ring with set of local units $\Sf{E}$}, provided that  for every finite subset
$\{r_1,\cdots,r_n\} \subset R$,
there exists $ e \in\Sf{E}$ such that
$$
e\,r_i\,\,=\,\, r_i\, e\,\,=\,\, r_i, \quad \text{ for every } \,
i=1,\cdots,n,
$$
In this way, given
a finite set $\{r_1,\cdots,r_n\}$ of elements in $R$,
we denote $$Unit\{r_1,\cdots,r_n\}\,\,=\,\, \left\{\underset{}{} e
\in\Sf{E} \, |
\,\, er_i= r_ie= r_i, \, \text{ for every } \, i=1,2,\cdots,n
\right\}.$$
Observe that our definition differs from that in  \cite[Definition 1.1]{Abrams:1983}, since we do not assume that the idempotents of $\mathsf{E}$ commute. In fact, our rings generalize those of \cite{Abrams:1983} since, when the idempotents are commuting, it is enough to require that
 for each element $r \in R$ there exists $e \in \Sf{E}$ such that $er\,=\, re\,=\, r$  (see  \cite[Lemma 1.2]{Abrams:1983}).  The Rees matrix rings considered in   \cite{Anh/Marki:1983}, are for instance, rings with a set of local units with noncommuting idempotents.

Let $R$ be a ring with a set of local units $\Sf{E}$. A right $R$-module $X$ is said to be \emph{unital} provided one of
the following equivalent conditions holds
\begin{enumerate}[(i)]
\item $X\tensor{R}R \cong X$ via the right $R$-action of $X$,
\item $XR:=\{\sum_{finite} x_ir_i|\, r_i \in R,\, x_i \in X\}=X$,
\item for every element $x \in X$, there exists an element $e \in
\Sf{E}$ such that $xe=x$.
\end{enumerate}
Left unital $R$-modules are analogously defined. 
Obviously any right $R$-module $X$ contains $XR$ as the largest
right unital $R$-submodule.

Let $R$ and $S$ be rings with a fixed set of local units,
respectively, $\Sf{E}$ and $\Sf{E}'$.
In the sequel, an \emph{extension of rings $\phi: (R, \Sf{E}) \to (S,\Sf{E}')$ with the same
set of local units},
stand for an additive and multiplicative map $\phi$ which satisfies
$\phi(\Sf{E})\,=\, \Sf{E}'$.
Note that since $R$ and $S$ have the same set of local units, any
right (resp. left) unital $S$-module
can be considered as right (resp. left) unital $R$-module by
restricting scalars.
Furthermore, for any right $S$-module $X$, we have the equality
$XR=XS$.

\section{Similar and invertible unital bimodules.}\label{sec:1}

Let $R$ be a ring with a set of local units $\mathsf{E}$. 
In the first part of this section we give some preliminary results
on  unital $R$-bimodules,
which we will use frequently in the sequel. Most of them
were already formulated by K. Hirata in \cite{Hirata:1968} and
\cite{Hirata:1969} when $R$ is unital.
For the convenience of the reader we will give in our case complete
proofs.
The second part provides a relation between the automorphism group
of an invertible unital
$R$--bimodule and the automorphisms of the base ring $R$. This result will be also used in the forthcoming
sections.

\smallskip

A \emph{unital $R$-bimodule} is an $R$-bimodule which is left and
right unital.
In the sequel, the expression \emph{$R$-bilinear map} means 
homomorphism of $R$-bimodules. Note that if $M$ is a unital
$R$-bimodule, then for every element $m \in M$, there exists a
\emph{two sided unit} $e$ for $m$. That is,
there exists $e \in  \Sf{E}$ such that  $m\,=\,me\,=em$, see \cite[p. 733]{ElKaoutit:2006}.
If several $R$-bimodules are handled, say $M, N, P$ with a finite
number of elements $\{ m_i, n_i, p_i\}_i$,
$m_i \in M, n_i \in N, p_i \in P$, then we denote by 
$Unit\{m_i, n_i, p_i\}$ the set of common two sided units of the
$m_i$'s, $n_i$'s and the $p_i$'s.
This means, that $e \in Unit\{m_i, n_i, p_i\} \subseteq \Sf{E}$ if
and only if
$$ em_i\,=\,m_i\,=\,m_i e,\,\, en_i\,=\,n_i\,=\,n_i e,\,\,
ep_i\,=\,p_i\,=\,p_i e,
\quad \text{ for every } i=1,\cdots,n.$$

Given a unital $R$-bimodule $M$, we denote its invariant
sub-bimodule,
that is, the  set of all $R$-bilinear maps  from $R$ to $M$, by 
$$ M^R\,\,:=\,\, \hom{R-R}{R}{M}.$$
It is canonically a module over the commutative ring of all
$R$-bimodule endomorphisms $\Z:=\Endo{R-R}{R}$. The $\Z$-action is
given as follows:
for every $z \in \Z$ and $f \in M^R$, we have 
$$ z.f: {}_RR_R \to {}_RM_R, \lr{r \mapsto
f(z(r))=z(e)f(r)=f(r)z(e),\,\, \text{where } e \in Unit\{r\} }.$$

\begin{lemma}\label{lema:1}
Let $R$ be a ring with local units and $M$ a unital $R$-bimodule
such that
$ \xymatrix{  M \ar@ ^{(->}^-{\oplus}[r] & R^{(n)}}$,
that is, $M$ is isomorphic as an $R$-bimodule to a direct summand of
direct sum of $n$ copies of $R$. Then
\begin{enumerate}[(i)]
 \item $M^R$ is a finitely generated and projective $\Z$-module;
\item there   is an isomorphism of $R$-bimodules
$$ M \,\, \cong\,\, R\tensor{\Z} M^R,$$ 
where the $R$-bimodule structure of the right hand term is induced
by that
of $R$. That is, 
$$ r(s\tensor{\Z}f) t \,\,=\,\, (rst)\tensor{\Z} f, \quad
(s\tensor{\Z}f  \in R\tensor{\Z}M^R, \,\, r,t \in R).$$
\end{enumerate}
\end{lemma}
\begin{proof}
Let us consider the following canonical $R$-bimodule homomorphisms:$$ \xymatrix{ \varphi_i: M \ar@ ^{(->}^-{\iota}[r] & R^{(n)} \ar@
{->>}^-{\pi_i}[r] & R, }
\quad \xymatrix{ \psi_i: R \ar@ ^{(->}^-{\iota_i}[r] & R^{(n)} \ar@
{->>}^-{\pi}[r] & M, } \quad i=1,\cdots, n,
$$ where $\pi_i$ and $\iota_i$ are the canonical projections and
injections.

$(i)$ It is clear that $\{(\psi_i,\, \psi_i^*)\}_i \in M^R \times
\hom{\Z}{M^R}{\Z}$ is a finite dual $\Z$-basis for $M^R$, where
each $\psi_i^*$ is defined by the right composition with
$\varphi_i$, that is, $\psi_i^*(f)\,=\, \varphi_i \circ f$,
for every $f \in M^R$.

$(ii)$ We know that there is a well defined $R$-bilinear map 
$$ \xymatrix@R=0pt{ R\tensor{\Z} M^R \ar@{->}^-{\eta}[rr] & & M \\
r\tensor{\Z} f \ar@{|->}[rr] & & f(r). }$$
We can easily show that  
$$ \xymatrix@R=0pt{ M \ar@{->}^-{\eta^{-1}}[rr] & & R\tensor{\Z}M^R
\\ m \ar@{|->}[rr] & & \sum_i^n \varphi_i(m)\tensor{\Z} \psi_i, }$$
is actually the inverse map of $\eta$.
\end{proof}

We will follow Miyashita's notation concerning this class of
$R$-bimodules.
That is, if we have an $R$-bimodule $M$
which is a direct summand as a bimodule of finitely many copies of
another $R$-bimodule $N$, we shall denote this situation by $M | N$.
It is clear that this relation is reflexive and transitive.
Furthermore, it is compatible with the tensor product
over $R$, in the sense that we have $M \tensor{R} Q | N\tensor{R} Q$
and $Q \tensor{R} M | Q\tensor{R} N$,
for every unital $R$-bimodule $Q$, whenever $M | N$.
In this way, we set 
\begin{equation}\label{Eq:similar}
M \,\,\sim \,\, N, \,\text{ if and only if, }\, M | N \text{ and } N
| M,
\end{equation}
and call $M$ \emph{is similar to} $N$. This in fact defines an
equivalence relation which is,
in the above sense, compatible with the tensor product over $R$.

\begin{lemma}\label{lema:2}
Let $R$ be a ring with local units and $M$, $N$ are unital
$R$-bimodules such that $M | R$ and $N|R$ as bimodules.
Then 
\begin{enumerate}[(i)]
\item for every finitely generated and projective $\Z$-module $P$,
we have an isomorphism of $\Z$-modules
$$ (R\tensor{\Z}P)^R \,\, \cong \,\, P;$$
\item there is a natural (on both components) $\Z$-linear
isomorphism
$$ (M\tensor{R}N)^R \,\, \cong \,\, M^R \tensor{\Z} N^R.$$
\end{enumerate}
\end{lemma}
\begin{proof}
$(i)$ Let $\{(p_i,p_i^*)\}_{1 \leq i \leq n}$ be a dual $\Z$-basis
for $P$. Given an element $f \in (R\tensor{\Z}P)^R$,
we define a finite family of elements
in $\Z$:
$$ z_{f,\, i} \,\,:=\,\, (R\tensor{\Z}p_i^*) \circ f: R\to
R\tensor{\Z}P \to R\tensor{\Z}\Z \cong R. $$
Now, an easy computation shows that the following maps are mutually
inverse
$$\xymatrix@R=0pt{ P \ar@{->}[rr] & & (R\tensor{\Z}P)^R \\ p
\ar@{|->}[rr] & & \left[ \underset{}{}
r\longmapsto r\tensor{\Z} p\right], } \qquad 
\xymatrix@R=0pt{ (R\tensor{\Z}P)^R \ar@{->}[rr] & & P \\ f
\ar@{|->}[rr] & & \sum_i^n z_{f,\, i} \, p_i, }$$
which gives the stated isomorphism. 

$(ii)$ By Lemma \ref{lema:1}(ii) we know that 
$$ M \,\, \cong\,\, R\tensor{\Z} M^R \text{ and } N \,\, \cong\,\,
R\tensor{\Z} N^R.$$ Therefore, we have
$$ M\tensor{R}N \,\, \cong R\tensor{\Z}\lr{M^R \,\tensor{\Z}\,
N^R}.$$
Since by Lemma \ref{lema:1}(i), $M^R \,\tensor{\Z}\, N^R$ is a
finitely generated and projective $\Z$-module,
we obtain the desired isomorphism by applying item $(i)$ just proved
above.
\end{proof}

Keep the notations of the first paragraph in the proof of Lemma
\ref{lema:1}.

\begin{proposition}\label{prop:1}
Let $R$ be a ring with local units and $M$, $N$ are unital
$R$-bimodules such that $M | R$ and $N|R$ as bimodules. Then there
is a natural
$R$-bilinear isomorphism $M\tensor{R}N \cong N\tensor{R} M$ given
explicitly by
$$\xymatrix@R=0pt{ \mathsf{T}_{M,N}\,: M \tensor{R}N
\ar@{->}^-{\cong}[rr] & & N\tensor{R} M
\\ x\tensor{R}y \ar@{|->}[rr] & & \sum_i^n \varphi_i(x)\, y
\tensor{R} \psi_i(e), }$$
where $e \in Unit\{x,y\}$. 
\end{proposition}
\begin{proof}
By Lemmata \ref{lema:1} and \ref{lema:2}, we have the following
chain of $R$-bilinear isomorphisms
$$\xymatrix{ M\tensor{R}N \ar@{->}^-{\cong}[r] & R\tensor{\Z}
M^R\tensor{\Z} N^R
\ar@{->}^-{\cong}[r] & R\tensor{\Z} N^R\tensor{\Z} M^R
\ar@{->}^-{\cong}[r] &
N\tensor{R} (R\tensor{\Z} M^R) \ar@{->}^-{\cong}[r] & N \tensor{R}M.
}$$
Writing down explicitly the composition of these isomorphisms, we
obtain the stated one.
\end{proof}

The natural transformation $\Sf{T}_{-,-}$ is associative in the
following sense.

\begin{lemma}\label{lema:3}
Let $R$ be ring with a local units. Consider $X, Y, P, Q, U, V$
unital
$R$-bimodules which satisfy $Q | R$, $(Y\tensor{R}V)|R$,
$(X\tensor{R}U) | R$ and $(P\tensor{R} Y)|R$, all as bimodules. Then
the following diagram
is commutative:
$$
\xymatrix@C=97pt@R=40pt{
X\tensor{R}U\tensor{R}P\tensor{R}Y\tensor{R}V\tensor{R}Q
\ar@{->}^-{X\tensor{}U\tensor{}P\tensor{}\Sf{T}_{Y\tensor{}V,\,Q}}[r]
\ar@{->}|-{\Sf{T}_{X\tensor{}U,\,P\tensor{}Y}\tensor{}V\tensor{}Q}[d]
& X\tensor{R}U\tensor{R}P\tensor{R}Q\tensor{R}Y\tensor{R}V
\ar@{->}|-{\Sf{T}_{X\tensor{}U,\,P\tensor{}Q\tensor{}Y}\tensor{}V}[d]
\\
P\tensor{R}Y\tensor{R}X\tensor{R}U\tensor{R}V\tensor{R}Q 
\ar@{->}_-{P\tensor{}\Sf{T}_{Y\tensor{}X\tensor{}U\tensor{}V,\,Q}}[r]
& P\tensor{R}Q\tensor{R}Y\tensor{R}X\tensor{R}U\tensor{R}V.}
$$  
\end{lemma}
\begin{proof}
Straightforward. 
\end{proof}

Recall form \cite[p. 227]{ElKaoutit/Gomez:2010} that a unital
$R$-bimodule $X$ is said to be \emph{invertible}
if there exists another unital $R$-bimodule $Y$ with two
$R$-bilinear isomorphisms
$$
\xymatrix{ X\tensor{R} Y \ar@{->}^-{\fk{l}}_-{\cong}[rr] & & R & &
Y\tensor{R} X. \ar@{->}_-{\fk{r}}^-{\cong}[ll] }
$$
As was shown in \cite{ElKaoutit/Gomez:2010}, one can choose the
isomorphisms $\fk{l}, \fk{r}$ such that
$$ X\tensor{R}\fk{r}\,\,=\,\, \fk{l}\tensor{R} X,\quad
Y\tensor{R}\fk{l}\,\,=\,\, \fk{r}\tensor{R} Y.$$
In others words, such that $(\fk{l}, \fk{r}^{-1})$ is a Morita
context from $R$ to $R$.
In this way, $Y$ is isomorphic as an $R$-bimodule 
to the right unital part $X^*R$ of the right dual $R$-module
$X^*\,=\, \hom{-R}{X}{R}$ of $X$. This isomorphism
is explicitly given by 
$$
Y \overset{\cong}{\longrightarrow} X^*R, \quad \lr{y \longmapsto [x
\mapsto \fk{r}(y\tensor{}x)]}.
$$
It is worth mentioning that $X$ is not necessarily finitely
generated as right unital $R$-module;
but a direct limit of finitely generated and projective right unital
$R$-modules.

Given $e \in \Sf{E}$ consider its decomposition with respect to the
invertible unital $R$-bimodule $X$, that is,
$$ e \,\,=\,\, \sum_{(e)} x_e y_e,\quad (\text{i.e.}\;\;  e \,\,=\,\, \sum_{(e)}  \fk{l}(x_e\tensor{R} y_e) ). $$
It is clear that the $x_e$'s and the $y_e$'s can be chosen such that$$ ex_e\,\,=\,\,x_e, \text{ and } y_e\,\,=\,\, y_e e,$$ which means that  $e$
is a right unit for the $y_e$'s and
left unit for the $x_e$'s.

\begin{remark}
It is noteworthy that a two sided unit of the set of elements $\{
x_e, y_e\}$ do not need
in general to coincide with $e$. 
However, any two sided unit $e_1 \in Unit\{ x_e, y_e\}$ is also a
unit for $e$. In other words,
the elements $e$ and $\sum_{(e)} x_e \,e \, y_e$ belong to the same
unital ring $eRe$, but they are not necessarily
equal. This makes a difference in technical difficulties
compared with the case of unital rings.
\end{remark}

The following lemma gives explicitly  the  relation between the automorphisms group of unital invertible $R$-bimodule and the unit group of the commutative unital ring $\Z=\Endo{R-R}{R}$.

\begin{lemma}\label{lema:4}
Let $R$ be a ring with local units and $X$ an invertible unital
$R$-bimodule. Then there is an isomorphism of groups, from the group of
$R$-bilinear
automorphisms of $X$, $\Aut{R-R}{X}$ to the unit group of $\Z$,
defined by
\begin{equation}\label{Eq:tilda}
\xymatrix@R=0pt{ \Aut{R-R}{X} \ar@{->}^-{\widetilde{(-)}}[rr] & &
\Aut{R-R}{R}=\mathcal{U}(\Z) \\ \sigma \ar@{|->}[rr] & &
\widetilde{\sigma} }
\end{equation}
where for each $r \in R$ with unit $e \in Unit\{r\}$, we have 
$$\widetilde{\sigma}(r) \,\,=\,\, \sum_{(e)} \sigma(x_e)y_e r
\,\,=\,\, \sum_{(e)} r\sigma(x_e)y_e.$$
In particular, for every element $\fk{t} \in X$ and $\sigma \in
\Aut{R-R}{X}$, we have
$$ \sigma(\fk{t})\,\,=\,\, \widetilde{\sigma}(e) \fk{t}, \quad \text{whenever
} e \in Unit\{\fk{t}\}.$$
\end{lemma}
\begin{proof}
Follows directly from \cite[Lemma 2.1]{Beattie/Del Rio:1999}.
\end{proof}

The Picard group of $R$ is denoted by $\Picar{R}$. It consists   of all isomorphisms classes of  invertible unital $R$-bimodules, with multiplication induced by the tensor product.  As in the unital case \cite[Theorem 2.(i)]{Froehlich:1973}, there is a canonical homomorphism of groups ${ \boldsymbol{\alpha}}: \Picar{R} \to \aut{\U(\Z)}$. This homomorphism is  explicitly given as follows.  Given $[X] \in \Picar{R}$ and  $u \in \U(\Z)$, we consider the $R$-bilinear automorphism of $X$  defined by $$\sigma_u:  X \longrightarrow X,\qquad \lr{\fk{t} \longmapsto \fk{t}\, u(e), \,\, \fk{t} \in X, \text{ and } e \in Unit\{\fk{t}\} }.$$ We clearly have $\sigma_{uv}\,=\, \sigma_{u} \circ \sigma_v$, and so $\widetilde{\sigma_{uv}}\,=\, \widetilde{\sigma_{u}} \circ \widetilde{\sigma_v}$, for $u,v \in \U(\Z)$
It follows from Lemma \ref{lema:4} that 
\begin{equation}\label{Eq:permutar}
\sigma_u(r\fk{t}) = \widetilde{\sigma_u}(r)\fk{t},
\text{ for every }\fk{t} \in X, r \in R.
\end{equation}
This equation implies that the element $\widetilde{\sigma_u} \in \U(\Z)$ is independent from the choice of the representative $X$ of the class $[X] \in \Picar{R}$.
We have thus a well defined map
\[
\boldsymbol{\alpha} : \Picar{R} \to \aut{\U(\Z)}, \qquad [X] \mapsto \boldsymbol{\alpha}_{[X]}(u)\,=\,\widetilde{\sigma_u}
 \] 
 which is a homomorphism of groups by \eqref{Eq:permutar}. This is exactly the homomorphism  induced by the map $\alpha$ of \cite[p. 135]{Beattie/Del Rio:1999}.  We have  the formula:
\begin{equation}\label{Eq:formula}
\boldsymbol{\alpha}_{[X]}(u)(r)\,\,=\,\,\widetilde{\sigma_u}(r)\,\,=\,\,\sum_{(e)} rx_e u(e_1) y_e, \quad (\text{i.e.}  \sum_{(e)} r\fk{l}(x_e\tensor{R} u(e_1) y_e)), 
\end{equation}
where $e \in Unit\{r\}$, $e_1 \in Unit\{x_e, y_e \}$, and as before $e\,=\, \sum_{(e)}x_ey_e$.
\color{black}

\section{Generalized crossed products with local units.}\label{sec:2}

Let $\G$ be any group with neutral element $1$. Consider a ring $R$
with a set of local units $\mathsf{E}$, and a group homomorphism
$\underline{\Theta} : \G \to \Picar{R}$. This is equivalent to give
sets of invertible $R$-bimodules and isomorphisms of bimodules
\[
\{ \Theta_x : x \in \G \}, \qquad \{ \F^{\Theta}_{x,y} : \Theta_x
\tensor{R} \Theta_y \overset{\cong}{\longrightarrow} \Theta_{xy},\,\,
x,y \in \G \},
\]
where $\underline{\Theta}(x) = [\Theta_x]$ for every $x \in \G$.
Consider the product on the direct sum $\bigoplus_{x \in \G}\Theta_x$ determined by
the maps $\F^{\Theta}_{x,y}$. There is no reason to expect that this
product, after the choice of the representatives $\Theta_x$, will be
associative. In fact, it becomes an associative multiplication if
and only if the diagram
\begin{equation}\label{asociativa}
\xymatrix@R=30pt{\Theta_x \tensor{R} \Theta_y \tensor{R} \Theta_z
\ar^-{\F^{\Theta}_{xy} \tensor{} \Theta_z}[rr] \ar_-{\Theta_x
\tensor{} \F^{\Theta}_{yz}}[d]& & \Theta_{xy} \tensor{R} \Theta_z
\ar^-{\F^{\Theta}_{xy,z}}[d]\\
\Theta_{x} \tensor{R} \Theta_{yz} \ar^-{\F^{\Theta}_{x,yz}}[rr] & &
\Theta_{xyz}}
\end{equation}
commutes for every $x,y,z \in \G$. Even in this case, it is not clear wether
this multiplication has a set of local units. On the other hand, we
know that there is an isomorphism of $R$-bimodules $\iota : R \to
\Theta_1$, so a natural candidate for such a set should be $\Sf{E}'
= \iota(\Sf{E})$. In fact, this is a set of local units if and only if, for every $x \in \G$, the following diagrams are commutative
\begin{equation}\label{Eq:unitario}
\xymatrix{\Theta_x \tensor{R} R \ar^-{\simeq}[rr] \ar_-{\Theta_x
\tensor{} \iota}[dr]& & \Theta_x \\
& \Theta_x \tensor{R} \Theta_1 \ar_-{\F^{\Theta}_{x,1}}[ur]& },
\qquad
\xymatrix{R \tensor{R} \Theta_x \ar^-{\simeq}[rr] \ar_-{\iota
\tensor{} \Theta_x}[dr]& & \Theta_x \\
& \Theta_1 \tensor{R} \Theta_x \ar_-{\F^{\Theta}_{1,x}}[ur]& }  
\end{equation}

The previous discussion suggest then  the following definition.

\begin{definition}\label{factormap}
A \emph{factor map} for the morphism $\underline{\Theta}: \G \to \Picar{R}$ consists of  sets of invertible $R$-bimodules and isomorphisms of bimodules
\[
\{ \Theta_x : x \in \G \}, \qquad \{ \F^{\Theta}_{x,y} : \Theta_x
\tensor{R} \Theta_y \overset{\cong}{\longrightarrow} \Theta_{xy},\,\,
x,y \in \G \},
\]
where $\underline{\Theta}(x) = [\Theta_x]$ for every $x \in \G$, and one isomorphism of $R$--bimodules $\iota : R \rightarrow \Theta_1$ such that the diagrams \eqref{asociativa} and \eqref{Eq:unitario} are commutative. 
\end{definition}

Now we come back to our initial homomorphism of groups $\underline{\Theta}: \G \to \Picar{R}$ and its associated family of isomorphisms $\{\F^{\Theta}_{x,y}\}_{x, y, \in \, \G}$. Our next aim is to find conditions under which we can extract from these data a factor map for $R$ and $\G$. First of all we should assume that the diagrams of  \eqref{Eq:unitario}  are commutative. We can define using the homomorphism of groups $\boldsymbol{\alpha}: \Picar{R} \to \aut{\U(\Z)}$ given by equation \eqref{Eq:formula}, a left $\G$-action on the group of units $\U(\Z)$, where $\Z={\rm End}_{R-R}(R)$.  Written down explicitly the composition $\boldsymbol{\alpha} \circ  \underline{\Theta}: \G \to \Picar{R} \to \aut{\U(\Z)}$ the outcome action is described as follows. Given a unit $e \in \Sf{E}$, its  decomposition  with respect to the invertible
$R$-bimodule $\Theta_x$ is written as 
$e\,=\, \sum_{(e)} x_e \bara{x}_e \in \iota^{-1}(\Theta_{1}) =R$. Using formula \eqref{Eq:formula},  the left $\G$-action on $\mathcal{U}(\Z)$ is then defined  by 
\begin{equation}\label{Eq:action}
\xymatrix@R=0pt{ \G \times \mathcal{U}(\Z) \ar@{->}[rr] &
&\mathcal{U}(\Z)
\\ (x, u) \ar@{->}[rr] & & {}^xu\,=\,\left[ r \mapsto \sum_{(e)} r
x_e u(e_1) \bara{x}_e\right], }
\end{equation}
where $e \in Unit\{r\}$ and $e_1 \in Unit\{x_e, \bara{x}_e \}$. 
Recall that in this case $e_1 \in Unit\{ e\}$, and so $e_1 \in
Unit\{r\}$.

The abelian group of all $n$-cocycles with respect to this action is
denoted by $Z_{\Theta}^n(\G,\mathcal{U}(\Z))$; while by
$B_{\Theta}^n(\G,\mathcal{U}(\Z))$ we denote the $n$-coboundary subgroup. The
$n$-cohomology group, is then denoted by
$H_{\Theta}^n(\G, \U(\Z))$.

\begin{remark}\label{rem:H}
To each homomorphism of groups $\underline{\Theta}:\G \to \Picar{R}$, it
corresponds as before a
$\G$-group structure on $\U(\Z)$, and so it leads to the cohomology
groups $H^*_{\Theta}(\G, \U(\Z))$.
The nature of this correspondence stills, up to our knowledge,  unexplored, even in the case of unital rings. However, as we will see, it is in fact the starting key point of any theory of generalized crossed products. Following the commutative and unital case, see Remark \ref{rem:com} below, here  
we will also extract our results from one fixed $\G$-action, this is to say 
fixed cohomology groups. For instance, one could start with the cohomology which corresponds to  the trivial homomorphism of groups, that is, $\mathsf{I}: \G \to \Picar{R}$ sending $x \mapsto [R]$, of course here the cohomology is not at all trivial.
\end{remark}

Here is a non trivial example.
\begin{example}\label{Example:B-Rio}(\cite{Beattie/Del Rio:1999})
Let $R$ be a ring with $\Sf{E}$ as a set of local units. Assume that there is a homomorphism of groups $\Phi : \G \to \aut{R}$ to the group of ring automorphisms of $R$ (we assume that $\Phi(x)(\Sf{E})\,=\, \Sf{E}$, for every $x \in \G$). Composing with the canonical homomorphism $\aut{R} \to \Picar{R}$ which sends $x \mapsto [R^x]$ (i.e. the $R$-bimodule whose underlying left $R$-module is ${}_{R}R$ and its right module structure is $r.s\,=\, rx(s)$), we obtain a homomorphism $\underline{\Theta}: \G \to \Picar{R}$. The corresponding family of maps $\{\F_{x,y}^{\Theta}\}_{x,\, y \in \G}$ is given by 
$$
\F_{x,y}^{\Theta}: R^x\tensor{R}R^y \longrightarrow R^{xy}, \left( r\tensor{R}s \longmapsto rx(s)\right).
$$
Here of course the associativity  and unitary  properties of the multiplication of $R$, both imply  that  $\{\F_{x,y}^{\Theta}\}_{x,\, y \in \G}$ is actually a factor map for $R$ and $\G$.  A routine computation shows now that the cohomology  $H^*_{\Theta}(\G,\U(\Z))$ which corresponds to this $\underline{\Theta}$ is exactly the one considerd in  \cite[p. 135]{Beattie/Del Rio:1999}, since the $\G$-action  \eqref{Eq:action} coincides for this case  with that given in  \cite{Beattie/Del Rio:1999}.
\end{example}

\begin{lemma}\label{lema:5}
Keep the above $\G$-action on $\mathcal{U}(\Z)$. Fix an element $x
\in \G$. 
\begin{enumerate}[(1)]
\item For every $\fk{t} \in \Theta_x$ and $u \in \mathcal{U}(\Z)$,
we have
\begin{equation}\label{Eq:t}
\fk{t} \, u(e) \,\,=\,\, {}^xu(e)\, \fk{t}, \quad \text{ where } e
\in Unit\{\fk{t}\}.
\end{equation}
\item Given similar bimodules $ M \sim \Theta_x$;
for every $\fk{m} \in M$ and $u \in \mathcal{U}(\Z)$, we have 
\begin{equation}\label{Eq:t1}
\fk{m} \, u(e) \,\,=\,\, {}^xu(e)\, \fk{m}, 
\quad \text{ with } e \in Unit\{\fk{m}\}. 
\end{equation} 
\end{enumerate}
\end{lemma}
\begin{proof}
$(1)$ This is equation \eqref{Eq:permutar}.\\
$(2)$  The identity \eqref{Eq:t} clearly lifts to finite direct sums $\Theta_x^{(n)}$, and, hence, for subbimodules $M$ of $\Theta_x^{(n)}$. \color{black} 
\end{proof}

\begin{proposition}\label{prop:3}
Let $R$ be a ring with a set of local units $\Sf{E}$ and $\G$ any group. Assume that
there is a  homomorphism of groups $\underline{\Theta}: \G \to \Picar{R}$ 
whose  family of maps $\{\F^{\Theta}_{x,y}\}_{x, y\,\in \G}$ satisfy the commutativity of diagrams \eqref{Eq:unitario}.  Denote by $\alpha_{x,y,z}: \Theta_{xyz} \to \Theta_{xyz}$ the isomorphisms that satisfy
\begin{equation}\label{Eq:ass1}
\alpha_{x,y,z} \circ \F^{\Theta}_{x, yz} \circ
\lr{\Theta_x\tensor{R}\F^{\Theta}_{y,z}} \,\,=\,\, \F^{\Theta}_{xy,z} \circ
\lr{\F^{\Theta}_{x,y}\tensor{R}\Theta_z}.
\end{equation}
Then $\widetilde{\alpha_{x,y,z}}$ defines a normalized element of $Z_{\Theta }^3(\G,\,
\mathcal{U}(\Z))$,
where $\widetilde{\alpha_{x,y,z}}$ is the image of $\alpha_{x,y,z}$
under the isomorphism
of groups stated in Lemma \ref{lema:4}. Furthermore, if $\widetilde{\alpha_{-,-,-}}$ is cohomologically trivial,   that is, there is a normalized map
$\sigma_{-,-} : \G\times\G \to \mathcal{U}(\Z)$
such that
$$\widetilde{\alpha _{x,y,z}}\,\,=\,\, \sigma_{x,yz}\, {}^x\sigma_{y,z}\,
\sigma_{x,y}^{-1}\, \sigma_{xy,z}^{-1},$$
for every $x, y,z \in\G$, then there is a factor map  $\{\bara{\F}^{\Theta}_{x,y}\}_{x.y \,\in\G}$ for $\underline{\Theta}$, defined by 
$$
\xymatrix@R=0pt{ \bara{\F}^{\Theta}_{x,y}: \Theta_x \tensor{R}\Theta_y
\ar@{->}[rr] & & \Theta_{xy} \\ \fk{u}_x\tensor{R}\fk{u}_y
\ar@{|->}[rr] & &
\sigma_{x,y}(e) \F^{\Theta}_{x,y}(\fk{u}_x\tensor{R}\fk{u}_y), }
$$
where $e \in Unit\{\fk{u}_x,\,\fk{u}_y\}$.
\end{proposition}
\begin{proof}
Consider $x, y, z, t \in\G$, and  
$\fk{u}_x \in \Theta_x$, $\fk{u}_y \in \Theta_y$, $\fk{u}_z \in
\Theta_z$ and $\fk{u}_t \in \Theta_t$, denote by
$\fk{u}_x\fk{u}_y$ the image of $\fk{u}_x\tensor{R}\fk{u}_y$ under
$\F^{\Theta}_{x,y}$.
There are several 
isomorphisms from
$\Theta_x\tensor{R}\Theta_y\tensor{R}\Theta_z\tensor{R}\Theta_t$ to
$\Theta_{xyzt}$
which are defined by the family $\F^{\Theta}_{-,-}$. So it will be convenient
to adopt some notations.
In order to do this, we rewrite the stated equation \eqref{Eq:ass1}
satisfied by the $\alpha_{-,-,-}$'s, as follows
$$
\alpha_{x,y,z}\lr{\fk{u}_x(\fk{u}_y\fk{u}_z)}\,\,=\,\,
\lr{(\fk{u}_x\fk{u}_y)\fk{u}_z}.
$$
Using Lemma \ref{lema:4}, this is equivalent to 
$$
\beta_{x,y,z}(e) \lr{\fk{u}_x(\fk{u}_y\fk{u}_z)}\,\,=\,\,
\lr{(\fk{u}_x\fk{u}_y)\fk{u}_z},
$$
where $e \in Unit\{\fk{u}_x, \fk{u}_y, \fk{u}_z\}$ and
$\beta_{-,-,-}:= \widetilde{\alpha_{-,-,-}}$
is the image of $\alpha_{x,y,z}$ under the homomorphism of groups
defined in Lemma \ref{lema:4}.
In this way, if we fix a unit $e \in Unit\{\fk{u}_x, \fk{u}_y,
\fk{u}_z, \fk{u}_t\}$, then we have
\begin{eqnarray*}
\lr{(\fk{u}_x\fk{u}_y)\fk{u}_z }\fk{u}_t &=& \beta_{x,y,z}(e)
\lr{\fk{u}_x(\fk{u}_y\fk{u}_z) }\fk{u}_t \\
&=& \beta_{x,y,z}(e)\beta_{x,yz,t}(e)
\lr{\fk{u}_x\lr{(\fk{u}_y\fk{u}_z) \fk{u}_t}} \\
&=& \beta_{x,y,z}(e)\beta_{x,yz,t}(e)
\lr{\fk{u}_x\lr{\beta_{y,z,t}(e)\lr{\fk{u}_y(\fk{u}_z \fk{u}_t)}}}
\\
&=& \beta_{x,y,z}(e)\beta_{x,yz,t}(e) \lr{\fk{u}_x
\beta_{y,z,t}(e)\lr{\fk{u}_y(\fk{u}_z \fk{u}_t)}}\\
&\overset{\eqref{Eq:t}}{=}& \beta_{x,y,z}(e)\beta_{x,yz,t}(e)
\lr{{}^x\beta_{y,z,t}(e)\fk{u}_x \lr{\fk{u}_y(\fk{u}_z \fk{u}_t)}}\\
&=& \beta_{x,y,z}(e)\beta_{x,yz,t}(e)\, {}^x\beta_{y,z,t}(e)
\lr{\fk{u}_x \lr{\fk{u}_y(\fk{u}_z \fk{u}_t)}}\\
&=& \beta_{x,y,z}(e)\beta_{x,yz,t}(e)\, {}^x\beta_{y,z,t}(e)
\beta^{-1}_{x,y,zt}(e)\lr{(\fk{u}_x \fk{u}_y)(\fk{u}_z \fk{u}_t)}\\
&=& \beta_{x,y,z}(e)\beta_{x,yz,t}(e) \,{}^x\beta_{y,z,t}(e)
\beta^{-1}_{x,y,zt}(e) \beta^{-1}_{xy,z,t}(e)
\lr{(\fk{u}_x \fk{u}_y)\fk{u}_z} \fk{u}_t.
\end{eqnarray*}
Given $h \in\Sf{E}$, using the above notations, we can consider the following units
\begin{eqnarray*}
h \,=\, \sum_{(h)}x_{h}\bara{x}_{h}, & & h_1 \in Unit\{x_{h}, \bara{x}_{h}\};  \\
h_1 \,=\, \sum_{(h_1)}y_{h_1}\bara{y}_{h_1}, & & h_2 \in Unit\{y_{h_1},\bara{y}_{h_1}\};  \\
h_2 \,=\, \sum_{(h_2)}z_{h_2}\bara{z}_{h_2}, & & h_3 \in Unit\{z_{h_2},\bara{z}_{h_2}\};  \\
h_3 \,=\, \sum_{(h_3)}t_{h_3}\bara{t}_{h_3}, & & h_4 \in Unit\{t_{h_3},\bara{t}_{h_3}\} . 
\end{eqnarray*}
It is clear that $h_4$ is a  unit for all the handled elements, in particular  $h_4 \in Unit\{x_h, y_{h_1}, z_{h_2}, t_{h_3}\}$. According to the equality showed previously, for every set $\{x_h, y_{h_1}, z_{h_2}, t_{h_3}\}$, we have
$$
\lr{(x_{h}y_{h_1})z_{h_2} }t_{h_3}\,\,=\,\,
\beta_{x,y,z}(h_4)\beta_{x,yz,t}(h_4) \,{}^x\beta_{y,z,t}(h_4)
\beta^{-1}_{x,y,zt}(h_4) \beta^{-1}_{xy,z,t}(h_4)
\lr{(x_{h} y_{h_1})z_{h_2}} t_{h_3}.
$$
Using the map $\F^{\Theta}_{t,t^{-1}}$, and summing up, we get 
$$
\sum_{(h_3)}\lr{(x_{h}y_{h_1})z_{h_2} }t_{h_3}\bara{t}_{h_3}\,=\,
\beta_{x,y,z}(h_4)\beta_{x,yz,t}(h_4) \,{}^x\beta_{y,z,t}(h_4)
\beta^{-1}_{x,y,zt}(h_4) \beta^{-1}_{xy,z,t}(h_4)
\lr{(x_{h} y_{h_1})z_{h_2}} t_{h_3}\bara{t}_{h_3},
$$
which means that 
$$
(x_{h}y_{h_1})z_{h_2} \,\,=\,\,
\beta_{x,y,z}(h_4)\beta_{x,yz,t}(h_4) \,{}^x\beta_{y,z,t}(h_4)
\beta^{-1}_{x,y,zt}(h_4) \beta^{-1}_{xy,z,t}(h_4)
(x_{h} y_{h_1})z_{h_2}, \,\,\, \text{equality in } \Theta_{xyz}.
$$
Repeating thrice the same process, we end up with
\begin{eqnarray*}
h &=&
\beta_{x,y,z}(h_4)\beta_{x,yz,t}(h_4) \,{}^x\beta_{y,z,t}(h_4)
\beta^{-1}_{x,y,zt}(h_4) \beta^{-1}_{xy,z,t}(h_4) h \\
&=& \beta_{x,y,z}(h)\beta_{x,yz,t}(h) \,{}^x\beta_{y,z,t}(h)
\beta^{-1}_{x,y,zt}(h) \beta^{-1}_{xy,z,t}(h),
\end{eqnarray*}
since $h_4h=hh_4=h$ and the involved maps are $R$-bilinear. Therefore,  for any unit $h \in \Sf{E}$, we have
$$
\beta_{xy,z,t} \circ \beta_{x,y,zt}(h)\,\,=\,\,{}^x\beta_{y,z,t}
\circ \beta_{x,yz,t} \circ \beta_{x,y,z}(h).
$$
Since the $\beta_{-,-,-}$'s are $R$-bilinear, this equality implies
the $3$-cocycle condition, that is,
$$
\beta_{xy,z,t} \circ \beta_{x,y,zt}\,\,=\,\,{}^x\beta_{y,z,t} \circ
\beta_{x,yz,t} \circ \beta_{x,y,z}, \quad \text{for every } x,y,z,t
\in\G.
$$
A routine computation, using equation \eqref{Eq:t}, shows the last statement.
\end{proof}

\begin{corollary}\label{coro:11}
Let $R$ be a ring with local units and $\G$ any group, and fix a factor map $\{\Theta_x\}_{x\, \in \G}$ for $R$ and $\G$ with its associated cohomology groups. Assume that there is a family of invertible unital $R$-bimodules
$\{\Omega_x\}_{x\, \in\, \G}$ such that $\Omega_x \sim\Theta_x$
(they are similar), for every $x \in \G$,
with a family of $R$-bilinear isomorphisms
$$\F_{x,y}: \Omega_x \tensor{R} \Omega_y \longrightarrow
\Omega_{xy}.$$
For $x, y, z \in \G$, we denote by $\alpha_{x,y,z}: \Omega_{xyz} \to
\Omega_{xyz}$ the resulting isomorphisms which satisfy
\begin{equation}\label{Eq:ass}
\alpha_{x,y,z} \circ \F_{x, yz} \circ
\lr{\Omega_x\tensor{R}\F_{y,z}} \,\,=\,\, \F_{xy,z} \circ
\lr{\F_{x,y}\tensor{R}\Omega_z}.
\end{equation}
Then $\widetilde{\alpha_{x,y,z}}$ defines an element of $Z_{\Theta}^3(\G,\,
\mathcal{U}(\Z))$,
where $\widetilde{\alpha_{x,y,z}}$ is the image of $\alpha_{x,y,z}$
under the isomorphism
of groups stated in Lemma \ref{lema:4}.
\end{corollary}
\begin{proof}
This is a direct consequence of Proposition \ref{prop:3} and Lemma \ref{lema:5}(2).
\end{proof}

We need to clarify the situation when two different classes of representatives are chosen.

\begin{proposition}\label{prop:2}
Let $R$ be a ring with a set of local units $\Sf{E}$ and $\G$ any group. Consider  two factor maps  $\{\F^{\Theta}_{x,y}\}_{x,y\, \in \G}$ and $\{\F^{\Gamma}_{x,y}\}_{x,y\, \in \G}$ such that, for every
$x \in \G$, there exists an $R$-bilinear
isomorphism $\Sf{a}_x: \Gamma_x \overset{\cong}{\rightarrow}\Theta_x
$.
Denote by $\tau_{x,y}$, where $x, y \in \G$, the following
$R$-bilinear isomorphism
$$
\xymatrix{\tau_{x,y}: \Theta_{xy}
\ar@{->}^-{\F^{\Theta}_{x,y}{}^{-1}}[rr] & &
\Theta_x\tensor{R}\Theta_y \ar@{->}^-{\cong}[rr] & &
\Gamma_x\tensor{R}\Gamma_y \ar@{->}^-{\F^{\Gamma}_{x,y}}[rr]
& & \Gamma_{xy} \ar@{->}^-{\cong}[rr] & & \Theta_{xy}, }
$$ 
That is,  
$$\tau_{x,y} \circ \F^{\Theta}_{x,y} \circ
(\Sf{a}_x\tensor{R}\Sf{a}_y)\,\,=\,\,
\Sf{a}_{xy} \circ \F^{\Gamma}_{x,y}.$$ 
Then $\widetilde{\tau_{x,y}}$ defines a normalized element of
$Z_{\Theta}^2(\G,\, \mathcal{U}(\Z))$,
where $\widetilde{\tau_{x,y}}$ is the image of $\tau_{x,y}$ under
the isomorphism of groups
stated in Lemma \ref{lema:4}.
\end{proposition}
\begin{proof}
Fix a set of elements $x, y, z \in \G$. Assume that  we have 
\begin{equation}\label{Eq:h}
{}^x\widetilde{\tau_{x, yz}}(h)
\widetilde{\tau_{x,yz}}(h)\,\,=\,\,\widetilde{\tau_{x,y}}(h)
\widetilde{\tau_{xy,z}}(h), \quad \text{for every unit }
 h \in \Sf{E}. 
\end{equation}
So taking any element $r \in R$ with unit $f \in Unit\{r\}$, we
obtain
\begin{eqnarray*}
\widetilde{\tau_{x,yz}} \circ {}^x\widetilde{\tau_{x, yz}}(r) &=&
\widetilde{\tau_{x,yz}} \circ {}^x\widetilde{\tau_{x, yz}}(f r) \\
&=&
{}^x\widetilde{\tau_{x, yz}}(f) \widetilde{\tau_{x,yz}}(r) \\ &=&
{}^x\widetilde{\tau_{x, yz}}(f) \widetilde{\tau_{x,yz}}(f) r \\
&\overset{\eqref{Eq:h}}{=}& \widetilde{\tau_{x,y}}(f)
\widetilde{\tau_{xy,z}}(f) r \\ &=& \widetilde{\tau_{xy,z}}
\lr{\widetilde{\tau_{x,y}}(f)r}
\\ &=& \widetilde{\tau_{xy,z}} \circ \widetilde{\tau_{x,y}}(r),
\end{eqnarray*}
which shows that $\widetilde{\tau_{-,-}}$ is a $2$-cocycle. 
In fact $\widetilde{\tau_{-,-}}$ is by definition a normalized
$2$-cocycle.
Equation \eqref{Eq:h} is fulfilled by following analogue steps 
as in the proof of Proposition \ref{prop:3}. 
\end{proof}

\begin{proposition}\label{prop:31}
The hypothesis are that of Corollary \ref{coro:11}. Assume further that
there are
two families of invertible unital $R$-bimodules 
$\{\Omega_x\}_{x\, \in\, \G}$ and $\{ \Gamma_x\}_{x \, \in\, \G}$ as
in Corollary \ref{coro:11} (i.e. $\Gamma_x \sim \Theta_x \sim \Omega_x$),
together  with families of $R$-bilinear isomorphisms
$$\F^{\Omega}_{x,y}: \Omega_x \tensor{R} \Omega_y \longrightarrow
\Omega_{xy},\quad
\F^{\Gamma}_{x,y}: \Gamma_x \tensor{R} \Gamma_y \longrightarrow
\Gamma_{xy}. $$
Assume also that there are $R$-bilinear isomorphisms $\Sf{a}_x:
\Omega_x \to \Gamma_x$, $x \in \G$.
Consider the associated $3$-cocycles 
$\beta^{\Omega}_{-,-,-},\, \beta^{\Gamma}_{-,-,-} \in Z_{\Theta}^3(\G,\,
\mathcal{U}(\Z))$ given by
Proposition \ref{prop:3}. Then, 
$$
\beta^{\Omega}_{-,-,-} \circ \lr{\beta^{\Gamma}_{-,-,-}}^{-1} \in
B_{\Theta}^2(\G,\, \mathcal{U}(\Z)).
$$
That is, they are cohomologous, or
$[\beta^{\Omega}_{-,-,-}]\,\,=\,\, [\beta^{\Gamma}_{-,-,-}]$ in
$H_{\Theta}^3(\G,\, \mathcal{U}(\Z))$.
\end{proposition}
\begin{proof}
Let us denote by $\alpha^{\Omega}_{-,-,-}$ and
$\alpha^{\Gamma}_{-,-,-}$ the
$R$-bilinear isomorphisms satisfying equation \eqref{Eq:ass},
respectively, for
$\{\Omega_x\}_{x\, \in\, \G}$ and $\{ \Gamma_x\}_{x \, \in \G}$. 
For any pair $x, y \,\in \G$, we consider the $R$-bilinear
isomorphism
$\Sf{b}_{xy}:\Omega_{xy} \to \Omega_{xy}$ defined by the following
equation
\begin{equation}\label{Eq:ab}
\Sf{a}_{xy} \circ \Sf{b}_{xy} \circ \F^{\Omega}_{x,y} \,\,=\,\,
\F^{\Gamma}_{x,y} \circ
\lr{\Sf{a}_x\tensor{R}\Sf{a}_y}.
\end{equation}
We denote by $\gamma_{x,y}$, $x, y \in\G$, the image of
$\Sf{b}_{xy}$ under the homomorphism
of groups defined in Lemma \ref{lema:4}.

Fix $x, y, z \in \G$, and consider $(\Sf{t}_x, \Sf{t}_y, \Sf{t}_z) 
\in \Omega_x \times \Omega_y \times \Omega_z$ and 
$(\Sf{u}_x, \Sf{u}_y, \Sf{u}_z) 
\in \Gamma_x \times \Gamma_y \times \Gamma_z$ with a common unit $e
\in Unit\{ \Sf{t}_x, \Sf{t}_y, \Sf{t}_z,
\Sf{u}_x, \Sf{u}_y, \Sf{u}_z\}$. Using the notation of the proof of
Proposition \ref{prop:3}, we can write
\begin{eqnarray}
\beta^{\Omega}_{x,y,z}(e)\, \Sf{t}_x(\Sf{t}_y \Sf{t}_z) &=&
(\Sf{t}_x \Sf{t}_y) \Sf{t}_z. \label{Eq:betaT} \\
\beta^{\Gamma}_{x,y,z}(e)\, \Sf{u}_x(\Sf{u}_y \Sf{u}_z) &=&
(\Sf{u}_x \Sf{u}_y) \Sf{u}_z. \label{Eq:betaU} \\
\gamma_{x,y}(e)\, \Sf{a}_{xy}(\Sf{t}_x \Sf{t}_y) &=& \Sf{a}_x
(\Sf{t}_x) \Sf{a}_y(\Sf{t}_y). \label{Eq:gammaA}
\end{eqnarray}
A routine computation as in the proof of Proposition \ref{prop:3}
using equation \eqref{Eq:betaT}, \eqref{Eq:betaU}
and \eqref{Eq:gammaA}, show that, for every unit $h \in \Sf{E}$, we
have
$$
\lr{\beta^{\Omega}_{x,y,z}(h)}^{-1} \beta^{\Gamma}_{x,y,z}(h)
\,\,=\,\, \lr{\gamma_{x,yz}(h)}^{-1}
\lr{{}^x\gamma_{y,z}(h)}^{-1} \gamma_{x,y}(h) \gamma_{xy,z}(h).
$$
This implies that 
$$
\beta^{\Gamma}_{x,y,z}\circ \lr{\beta^{\Omega}_{x,y,z}}^{-1}
\,\,=\,\,
\gamma_{xy,z} \circ \gamma_{x,y} \circ \lr{{}^x\gamma_{y,z}}^{-1}
\circ \lr{\gamma_{x,yz}}^{-1},
\text{ in } \mathcal{U}(\Z),
$$
which finishes the proof.
\end{proof}

Now we are able to give the right definition of generalized crossed product.
We hope it will result cleaner than the one  
considered in \cite{Kanzaki:1968,Miyashita:1973} for  rings with identity.
In the sequel, $R$ denotes a ring with a set of local units $\Sf{E}$.

\begin{definition}
Given a factor map as in Definition \ref{factormap}
\[
\{ \F^{\Theta}_{x,y} : \Theta_x
\tensor{R} \Theta_y \overset{\cong}{\longrightarrow} \Theta_{xy},\,\,
x,y \in \G \},
\]
with an isomorphism of $R$-bimodules $\iota : R \rightarrow \Theta_1$, 
we define its associated \emph{generalized crossed product} $\Delta (\Theta) = \bigoplus_{x \in \G } \Theta_x$ with multiplication $$\theta_x \theta_y = \F^{\Theta}_{x,y} (\theta_x \tensor{} \theta_y), \text{  for } \theta_x \in \Theta_x, \theta_y \in \Theta_y.$$ 
\end{definition}

It follows that $\Theta_1$ is a subring of $\Delta (\Theta)$ and the map $\iota : R \to \Theta_1$ is a ring isomorphism. Therefore, $\iota (\mathsf{E})$ is a set of local units for $\Theta_1$ that serves as well as a set of local units for $\Delta (\Theta)$. Normally, we will identify $R$ with $\Theta_1$, and, thus, $\mathsf{E}$ will be a set of local units for the generalised crossed product $\Delta (\Theta)$.  Obviously, $\Delta (\Theta)$ is a $\G$--graded ring such that $\Theta_x \Theta_y = \Theta_{xy}$ for every $x, y \in \G$ (thus, it could be referred to as a \emph{strongly graded} ring after \cite{Nastasescu/Oystaeyen:82}).
For instance, if we take $\Theta$ as in Example \ref{Example:B-Rio}, then $\Delta(\Theta)$ is the  skew group ring $R*\G$, see \cite{Beattie/Del Rio:1999}.

\begin{remark}
A factor map as in Definition \ref{factormap} determines a
generalized crossed product $\Gamma = \bigoplus_{x \in \G}\Gamma_x$
with ring isomorphism $\iota: R \rightarrow \Gamma_1$ and
multiplication given by $\gamma_x \gamma_y =
\F^{\Gamma}_{x,y}(\gamma_x \tensor{R} \gamma_y)$ for $\gamma_x \in
\Gamma_x, \gamma_y \in \Gamma_y$. In other words, generalized
crossed products and factor maps are just two ways to define the same
mathematical object. A generalized crossed product $\Gamma$ gives clearly a
group homomorphism
\[
\xymatrix{\underline{\Gamma} : \G \ar[r] & \Picar{R}}, \qquad (x
\mapsto [\Gamma_x]).
\] 
Whether a general group homomorphism $\G \rightarrow \Picar{R}$ gives
some generalized crossed product is not so clear. However, as we have seen in Proposition \ref{prop:3}, this  is possible,  if the underlying maps satisfy the commutativity of diagrams \eqref{Eq:unitario} and the associated $3$-cocycle is trivial or at least cohomologically trivial.
\end{remark}

\begin{definition}
A \emph{morphism} of generalized crossed products $(\Delta(\Theta), \nu),
(\Delta(\Gamma), \iota)$ is a graded ring homomorphism $f: \Delta(\Theta) \to
\Delta(\Gamma)$ such that $f \circ \nu = \iota$. Equivalently, it consists
of a set of homomorphisms of $R$--bimodules
\[
\{ f_x : \Theta_x  \rightarrow \Gamma_x : x \in \G \}
\] 
such that all the diagrams
$$
\xymatrix@C=60pt{ \Theta_x\tensor{R}\Theta_y
\ar@{->}_-{f_x\tensor{}f_y}[d] \ar@{->}^-{\F^{\Theta}_{x,y}}[rr] & &
\Theta_{xy} \ar@{->}^-{f_{xy}}[d] \\ \Gamma_x\tensor{R}\Gamma_y
\ar@{->}_-{\F^{\Gamma}_{x,y}}[rr] & & \Gamma_{xy}}
$$
commute, and $f_1 \circ \nu = \iota$.  
\end{definition}

Let $R \subseteq S$ be  an extension of rings with the same set of local units $ \Sf{E}$. We denote as in \cite{ElKaoutit/Gomez:2010} 
by $\Inv{R}{S}$ the group of all invertible unital $R$-sub-bimodules of $S$. An $R$-sub-bimodule $X \subseteq S$ 
belongs to the group $\Inv{R}{S}$ if and only if there exists an $R$-sub-bimodule $Y \subseteq S$ such that 
$$
XY\,\,=\,\, R \,\, = \,\, YX,
$$
where the first and last terms are obviously defined  by the multiplication of $S$. Clearly, there is a homomorphism of groups $\mu: \Inv{R}{S} \to \Picar{R}$.

\begin{remark}\label{rem:triangulo}
A generalized crossed product $\Gamma$ of $R$ with $\G$ gives in particular the
ring extension $\iota : (R,E) \rightarrow (\Gamma,E')$ with local
untis. In this way, $\Gamma_x$ is a unital $R$-subbimodule of
$\Gamma$ for every $x \in \G$, and the isomorphism $\iota : R
\rightarrow \Gamma_1$ becomes $R$--bilinear. We will identify $\Gamma_1$ with $R$ and $\Sf{E}$ with $\Sf{E}'$. Let us denote by
$\F^{\Gamma} : \Gamma \tensor{R} \Gamma \rightarrow \Gamma$ the
multiplication map of the ring $\Gamma$. It follows from \cite[Lemma
1.1]{ElKaoutit/Gomez:2010} that its restriction to $\Gamma_x
\tensor{R} \Gamma_y$ gives an $R$-bilinear isomorphism
\[
\xymatrix{\F^{\Gamma}_{xy} : \Gamma_x \tensor{R} \Gamma_y
\ar^{\simeq}[r] & \Gamma_x\Gamma_y = \Gamma_{xy}}
\]
for every $x, y \in \G$, since, obviously, $\Gamma_x \in
\Inv{R}{\Gamma}$ for every $x \in \G$. Clearly, our generalized crossed product
determines a factor map in the sense of  Definition \ref{factormap}. On the other hand, the associated homomorphisms of groups $\underline{\Gamma}: \G \to \Picar{R}$ actually factors throughout the morphism $\mu$. That is, we have a commutative diagram of groups
\begin{equation}\label{Eq:factoriza}
\xymatrix@R=30pt{ \G \ar@{->}^-{\underline{\Gamma}}[rr]  \ar@{-->}_-{\bara{\Gamma}}[rd] & & \Picar{R} \\ & \Inv{R}{\Gamma}. \ar@{->}_-{\mu}[ru] &  }
\end{equation}
\end{remark}

As we have seen in Remark \ref{rem:triangulo}, one can start with a fixed ring extension $R\subseteq S$ with the same set of local units $\Sf{E}$, and then  work with the $\bara{\Gamma}$'s defined this time by $\bara{\Gamma}: \G \to \Inv{R}{S}$ instead of the $\underline{\Gamma}$'s.
In this case the situation becomes much more manageable, in the sense that  each morphism $\bara{\Gamma}$ induces automatically a factor map and vice versa.

\begin{remark}\label{rem:com}
As was shown in \cite[Remarks 1., 2. and 3.]{Kanzaki:1968}, 
the above notion generalizes the classical notion of crossed
product in the commutative Galois setting.  Specifically, given a commutative Galois extension $\Bbbk \subseteq
R$ with
a finite Galois group $\G$, and consider the canonical homomorphism
of groups
$\Phi_0: \G \to \RPicar{\Bbbk}{R}$, 
where $\RPicar{\Bbbk}{R}\,=\,\{ [P] \in \Picar{R}|\,\, kp=pk,
\forall \, k \in \Bbbk, p \in P\}$, which
sends $x \to [R^x]$ to the class of the $R$-bimodules $R^x$ whose
underlying left $R$-module is ${}_RR$ and
its right $R$-module structure was twisted by the automorphism $x$, see Example \ref{Example:B-Rio}. 
In this case, the over ring is chosen to be  $S\,=\, \oplus_{x \in \G}\Phi_0(x)$ 
defined as the usual crossed product attached to $\Phi_0$, 
which is not necessary trivial, that is, defined using any
$2$-cocycle.
Thus, in the Galois theory of commutative rings, the homomorphism of groups which leads to the study of intermediate crossed products rings, is the obvious homomorphism of groups $\bara{\Theta}$ which satisfies $\Phi_0\,=\,\mu \circ \bara{\Theta}$, where $\mu$ is as in diagram \eqref{Eq:factoriza}.
\end{remark}

\section{An abelian group of generalized crossed products.}\label{sec:3}

From now on, we fix a ring extension $R \subseteq S$ with the same set of local units $\Sf{E}$. Let $\G$ be any group and assume given  a morphism of groups 
$$ \overline{\Theta}: \G \longrightarrow \Inv{R}{S},\quad \lr{x \longmapsto \Theta_x},$$
where we have used the notation of Remark \ref{rem:triangulo}.
The multiplication of $S$ induces then a factor map
$$\{ \F^{\Theta}_{x,y}: \Theta_x \tensor{R} \Theta_y \longrightarrow \Theta_{xy}, : x,\, y\, \in \G\}. $$ 
Clearly $R = \Theta_1$, whence $\mathsf{E}$ is a set of local units for $\Delta (\Theta)$. 
We introduce in this section an abelian group $\Scr{C}(\Theta/R)$
whose elements are the isomorphism classes of generalized crossed products
whose homogeneous components are similar as $R$-bimodules to that of
$\Delta(\Theta)$, see Section \ref{sec:2} for definition.
Our definition is inspired from that of Miyashita 
\cite[\S 2.]{Miyashita:1973}, however the proofs presented here are
slightly different. This group will be crucial
for the construction of the seven terms exact sequence of groups
in the next section.
\smallskip

 Let us consider thus
the set $\Scr{C}(\Theta/R)$ of isomorphism classes of generalized crossed products
$[\Delta(\Gamma)]$ such that, for each $x \in \G$, we have
$\Gamma_x \sim \Theta_x$, that is they are similar as $R$-bimodules, see
Section \ref{sec:1}.

\begin{proposition}\label{prop:4}
Consider the set of isomorphisms classes of
generalized crossed products $\Scr{C}(\Theta/R)$ defined
above.
Then $\Scr{C}(\Theta/R)$ has the structure of an abelian group.
Moreover the subset $\Scr{C}_0(\Theta/R)$ consisting
of classes $[\Delta(\Lambda)]$ such that for every $x \in \G$,
$\Lambda_x\cong \Theta_x$ as $R$-bimodules, is a sub-group of
$\Scr{C}(\Theta/R)$.
\end{proposition}
\begin{proof}
Given two classes $[\Delta(\Omega)], [\Delta(\Gamma)] \in
\Scr{C}(\Theta/R)$, their multiplication is defined by
$$ [\Delta(\Omega)] \, [\Delta(\Gamma)]\,\, =\,\, [\underset{x \in
\G}{\oplus}\, \Omega_x\tensor{R}\Theta_{x^{-1}}\tensor{R}\Gamma_x].
$$
The factor maps of its representative generalized crossed product
are given by the following composition
$$
\xymatrix{ \Omega_x\tensor{R}\Theta_{x^{-1}}\tensor{R}\Gamma_x
\tensor{R}
\Omega_y\tensor{R}\Theta_{y^{-1}}\tensor{R}\Gamma_y
\ar@{->}|-{\Omega_x\tensor{}
\Sf{T}_{\Theta_{x^{-1}}\tensor{}\Gamma_x,\,
\Omega_y\tensor{R}\Theta_{y^{-1}}}\tensor{}\Gamma_y}[rdd]
\ar@{-->}_-{\F^{\Omega\Gamma}_{x,y}}@/_8pc/[dddddr] & & \\ & & \\ &
\Omega_x\tensor{R}\Omega_{y}\tensor{R}\Theta_{y^{-1}} \tensor{R}
\Theta_{x^{-1}}\tensor{R}\Gamma_{x}\tensor{R}\Gamma_y 
\ar@{->}|-{\F^{\Omega}_{x,y}\tensor{}\F^{\Theta}_{y^{-1},x^{-1}}\tensor{}\F^{\Gamma}_{x,y}}[ddd]
&
\\ & & \\ & & \\ &
\Omega_{xy}\tensor{R}\Theta_{(xy)^{-1}}\tensor{R}\Gamma_{xy}, & }
$$
where $\Sf{T}_{\Theta_{x^{-1}}\tensor{}\Gamma_x,\,
\Omega_y\tensor{R}\Theta_{y^{-1}}}$ is the twist $R$-bilinear map
defined in Proposition \ref{prop:1}.
The associativity of the 
$\F^{\Omega\Gamma}_{-,-}$'s is easily deduced using Lemma
\ref{lema:3}.
This multiplication is commutative since for any two classes 
$[\Delta(\Omega)], [\Delta(\Gamma)] \in \Scr{C}(\Theta/R)$, we have
a family of $R$-bilinear
isomorphisms
$$ 
\xymatrix@R=40pt{ \Gamma_x \tensor{R} \Theta_{x^{-1}} \tensor{R}
\Omega_{x} \ar@{->}^-{\cong}[rr]
\ar@{-->}_-{\cong}@/_4pc/[ddrr] & & \Gamma_x \tensor{R}
\Theta_{x^{-1}}
\tensor{R} \Omega_{x}\tensor{R}\Theta_{x^{-1}} \tensor{R} \Theta_x
\ar@{->}|-{\Sf{T}\tensor{}\Theta_x}[d] \\ & &
\Omega_x \tensor{R} \Theta_{x^{-1}} 
\tensor{R} \Gamma_{x}\tensor{R}\Theta_{x^{-1}} \tensor{R} \Theta_x
\ar@{->}^-{\cong}[d] \\ & & \Omega_x \tensor{R} \Theta_{x^{-1}}
\tensor{R} \Gamma_{x},}
$$
which is easily shown to be compatible with both factors maps of
$\Gamma$ and $\Omega$.
Thus, it induces an isomorphism at the level of 
generalized crossed products.

It is clear that the class of $[\Delta(\Theta)]$ is the unit of 
this multiplication. The inverse of any class $[\Delta(\Omega)]$ is
given by
$$
[\Delta(\Omega)]^{-1}\,\,=\,\, [\underset{x \in\G}{\oplus}
\Theta_x\tensor{R}\Omega_{x^{-1}}\tensor{R} \Theta_x].
$$
Lastly, the fact that $\Scr{C}_0(\Theta/R)$ is a sub-group of
$\Scr{C}(\Theta/R)$, can be immediately deduced
from definitions.
\end{proof}

\begin{proposition}\label{prop:5}
Consider the abelian group
$\Scr{C}_0(\Theta/R)$ of Proposition \ref{prop:4}. Then there is an
isomorphism of groups
$$
\Scr{C}_0(\Theta/R) \,\, \cong\,\, H_{\Theta}^{2}(\G, \mathcal{U}(\Z)).
$$
\end{proposition}
\begin{proof}
The stated isomorphism is given by the following map:
$$ \zeta: \Scr{C}_0(\Theta/R) \longrightarrow H_{\Theta}^{2}(\G,
\mathcal{U}(\Z)), \quad \lr{ [\Delta(\Lambda)] \longmapsto
[\widetilde{\tau}_{-,-}] },$$
where $\widetilde{\tau}_{-,-}$ is the normalized $2$-cocycle defined
in Proposition \ref{prop:2} and $[\widetilde{\tau}_{-,-}]$ denotes
its
equivalence class in $H_{\Theta}^2(\G,\mathcal{U}(\Z))$.
First we need to show that this map is in fact well defined. To do
so, we take two isomorphic generalized crossed products $\chi:
\Delta(\Lambda)
\overset{\cong}{\to} \Delta(\Sigma)$ which represent the same
element in the subgroup
$\Scr{C}_0(\Theta/R)$, with $R$-bilinear isomorphisms 
$\Sf{a}_x:\Lambda_x \rightarrow \Theta_x \leftarrow \Sigma_x:
\Sf{b}_x$,
for all $x\in \G$. Then we check that the associated $2$-cocycles
are cohomologous. So let us denote these cocycles by
$\widetilde{\tau}_{-,-}$ and $\widetilde{\gamma}_{-,-}$ associated,
respectively, to $\Delta(\Lambda)$ and
$\Delta(\Sigma)$. We need to show that
$[\widetilde{\tau}_{-,-}]\,=\,[\widetilde{\gamma}_{-,-}]$.
Recall from Proposition \ref{prop:2} that we have 
$$
\tau_{x,y} \circ \F^{\Theta}_{x,y} \circ
\lr{\Sf{a}_x\tensor{R}\Sf{a}_y}\,\,=\,\, \Sf{a}_{xy} \circ
\F^{\Lambda}_{x,y}, \quad
\gamma_{x,y} \circ \F^{\Theta}_{x,y} \circ
\lr{\Sf{b}_x\tensor{R}\Sf{b}_y} \,\,=\,\, \Sf{b}_{xy} \circ
\F^{\Sigma}_{x,y},
$$
for every $x, y \in \G$, where $\F^{\Lambda}_{-,-}$ and
$\F^{\Sigma}_{-,-}$ are, respectively,
the factor maps relative to $\Lambda$ and $\Sigma$. For each $x \in
\G$, we set the following $R$-bilinear
isomorphism
$$
\xymatrix{\beta_x: \Theta_x \ar@{->}^-{\Sf{a}_x^{-1}}[r] & \Lambda_x
\ar@{->}^-{\chi_x}[r] & \Sigma_x \ar@{->}^-{\Sf{b}_x}[r] &
\Theta_x},
$$
where $\chi_x$ is the $x$-homogeneous component of the isomorphism
$\chi$.
Then it is easily seen, using the fact that $\chi$ is morphism of
generalized crossed products,
that, for each pair of elements $x,y\in\G$,  we have
\begin{equation}\label{Eq:beta}
\beta_{xy} \circ \tau_{x,y} \circ \F^{\Theta}_{x,y} \,\, =\,\, 
\gamma_{x,y} \circ \F^{\Theta}_{x,y} \circ
\lr{\beta_x\tensor{R}\beta_y}.
\end{equation}
We claim that 
\begin{equation}\label{Eq:beta-gamma}
\widetilde{\tau}_{x,y}(e)
\widetilde{\beta}_{xy}(e)\,\,=\,\,\widetilde{\gamma}_{x,y}(e)
\widetilde{\beta}_{x}(e) {}^x\widetilde{\beta}_y(e), \quad
\text{for all } e \in \Sf{E}.
 \end{equation}
So let $e \in \Sf{E}$, and consider its decomposition
$e\,=\,\sum_{(e)} x_e \bara{x}_e$
with respect to $\Theta_x$, for a given $x \in \G$. Let $e_1 \in
Unit\{x_e ,\bara{x}_e \}$.
Using equation \eqref{Eq:beta}, we get the following equality in $S$
\begin{equation}\label{Eq:iguales}
\sum_{(e),\, (e_1)} \beta_{xy} \circ \tau_{x,y} (x_ey_{e_1})
\bara{y}_{e_1}\bara{x}_e
\,\, =\,\, \sum_{(e),\, (e_1)} \gamma_{x,y}
\lr{\beta_x(x_e)\beta_y(y_{e_1})} \bara{y}_{e_1}\bara{x}_e.
\end{equation}
Routine computations show that,
\begin{eqnarray*}
\sum_{(e),\, (e_1)} \beta_{xy}\tau_{x,y}(x_ey_{e_1})
\bara{y}_{e_1}\bara{x}_e &=&
\widetilde{\tau}_{x,y}(e_1) \widetilde{\beta}_{xy}(e_1) e
\,\,\,=\,\,\, \widetilde{\tau}_{x,y}(e)
\widetilde{\beta}_{xy}(e),\,\,\text{ and } \\
\sum_{(e),\, (e_1)} \gamma_{xy}\lr{\beta_x(x_e)\beta_y(y_{e_1})}
\bara{y}_{e_1}\bara{x}_e &\overset{\eqref{Eq:t}}{=}&
\widetilde{\gamma}_{x,y}(e_1) \widetilde{\beta}_{x}(e_1)
{}^x\widetilde{\beta}_y(e_1) e\,\,\,=\,\,\,
\widetilde{\gamma}_{x,y}(e) \widetilde{\beta}_{x}(e)
{}^x\widetilde{\beta}_y(e),
\end{eqnarray*}
which by \eqref{Eq:iguales} imply that 
$$
\widetilde{\tau}_{x,y}(e)
\widetilde{\beta}_{xy}(e)\,\,=\,\,\widetilde{\gamma}_{x,y}(e)
\widetilde{\beta}_{x}(e) {}^x\widetilde{\beta}_y(e),
$$
which is the claimed equality. On the other hand we define 
$$
\xymatrix@R=0pt{h: \G \ar@{->}[rr] & &
\mathcal{U}(\Z)\,=\,\Aut{R-R}{R} \\ x \ar@{|->}[rr] & & \lr{h_x: r
\longmapsto \sum_{(e)}r\beta_x(x_{e}) \bara{x}_e} }
$$
where $e \in Unit\{r\}$, and $e\,=\, \sum_{(e)}x_e\bara{x}_e$ is the
decomposition of $e$ in $R\,=\, \Theta_x\Theta_{x^{-1}}$. It is
clear form
definitions that $\widetilde{\beta}_x\,=\, h_x$, for ever $x \in
\G$. Taking an arbitrary element $r \in R$, we can show using
equation
\eqref{Eq:beta-gamma} that 
$$ \widetilde{\tau}_{x,y} \circ h_{xy}(r)\,\,=\,\,
\widetilde{\gamma}_{x,y} \circ {}^xh_y \circ h_x(r).$$
Therefore, $\widetilde{\tau}_{x,y} \circ h_{xy}\,\,=\,\,
\widetilde{\gamma}_{x,y} \circ {}^xh_y \circ h_x,$
which means that $\widetilde{\tau}_{-,-}$ and
$\widetilde{\gamma}_{-,-}$ are cohomologous.

By the same way, we show that, if $\widetilde{\tau}_{-,-}$ and
$\widetilde{\gamma}_{-,-}$ are cohomologous, then
$\Delta(\Lambda/R)$ and $\Delta(\Sigma/R)$ are isomorphic as
generalized crossed products,
which  means that the stated map $\zeta$ is injective. 
To show that this map is surjective, we start with a given
normalized
$2$-cocycle $\sigma_{-,-} :\G^2 \to \mathcal{U}(\Z)$, 
and consider the following $R$-bilinear automorphisms: for every $x
\in \G$
$$ \vartheta_x: \Theta_x \longrightarrow \Theta_x,\quad \lr{\fk{u}
\longmapsto \sigma_{x,x}(e)\fk{u}},$$
where $e \in Unit\{\fk{u}\}$. Now set $\Sigma_x\,=\, \Theta_x$ as an
$R$-bimodule with the $R$-bilinear isomorphism
$\vartheta_x: \Sigma_x \to \Theta_x$ previously defined.
We define new factor maps relative to those $\Sigma_x$'s, by 
$$ \F^{\Sigma}_{x,y}: \Sigma_x\tensor{R}\Sigma_y \to
\Sigma_{xy},\quad \lr{\fk{u}_x\tensor{R}\fk{u}_y \longmapsto
\sigma_{x,y}(e)\F^{\Theta}_{x,y}(\fk{u}_x\tensor{R}\fk{u}_y)},$$
where $e \in Unit\{\fk{u}_x,\fk{u}_y\}$. 
The $2$-cocycle condition on $\sigma_{-,-}$ gives the associativity
of $\F^{\Sigma}_{-,-}$, while
the normalized condition gives the unitary property. Thus,
$\Delta(\Sigma/R)$
is a generalized crossed product whose isomorphic 
class belongs by definition to the sub-group $\Scr{C}_0(\Theta/R)$.

Lastly, by a routine computation, using the definition of the twist
natural map given in Proposition \ref{prop:1},
we show that the map $\zeta$ is a homomorphism of groups.
\end{proof}

\section{The Chase-Harrison-Rosenberg seven terms exact
sequence.}\label{sec:4}
In this section we show the analogue of Chase-Harrison-Rosenberg's
\cite{Chase et al:1965} seven terms exact
sequence, generalizing by this the case of commutative Galois
extensions with finite Galois group due of
T. Kanzaki \cite[Theorem p.187]{Kanzaki:1968} and 
the noncommutative unital case treated by Y. Miyashita in
\cite[Theorem 2.12]{Miyashita:1973}.
\smallskip

From now on we fix a ring extension $R \subseteq S$ with a same set
of local units $\Sf{E}$,
together with a morphism of groups $\overline{\Theta}: \G \to \Inv{R}{S}$ and
the associated generalized crossed product
$\Delta(\Theta/R)$  having $\F^{\Theta}_{-,-}$ as a
factor
map relative to $\Theta$. 

We denote by $\Picar{R}$ and $\Picar{S}$, respectively, the Picard
group of $R$ and $S$.
There is a canonical left $\G$-action on $\Picar{R}$ induced by the
obvious homomorphism of groups $\Inv{R}{S} \to \Picar{R}$ given explicitly by
$$
{}^x[P]\,\,=\,\,
[\Theta_x\tensor{R}P\tensor{R}\Theta_{x^{-1}}],\quad \text{ for
every } x \in \G, \text{ and } [P] \in\Picar{R}.
$$
The corresponding $\G$-invariant sub-group is denoted by
$\Picar{R}^{\G}$.
We will consider 
the sub-group $\RPicar{\Z}{R}$ whose elements are represented by
$\Z$-invariant $R$-bimodules, in the sense that
$$
[P] \in \RPicar{\Z}{R},\,\,\text{if and only if }\,\, [P] \in
\Picar{R}, \text{ and }\,
\fk{z}(e) p e'\,\,=\,\, e p \fk{z}(e'),\,\, \forall p \in P,
$$
for every pair of units $e,e' \in \Sf{E}$, and every element $\fk{z}
\in \Z$. In this way, we set
$$
\RPicar{\Z}{R}^{\G}\,\,:=\,\, \RPicar{\Z}{R} \,\cap\,
\Picar{R}^{\G}.
$$

Recall from \cite[Sect. 3]{ElKaoutit/Gomez:2010} the group $\Scr{P}(S/R)$:
$$
\Scr{P}(S/R)\,=\, \left\{ \underset{}{} \objeto{[P]}{\phi}{[X]}|\,
[P] \in \Picar{R}, [X] \in \Picar{S}, \text{ and a map }
\phi: {}_RP_R \to {}_RX_R  \right\}
$$ 
where at least one of the maps $\bara{\phi}_{l} : P\tensor{R}S \to
X$, or $\bara{\phi}_{r} : S\tensor{R}P \to X$,
canonically attached to $\phi$, is an isomorphisms of bimodules. For
more details on the structure group of this set we refer to
\cite[Section 3.]{ElKaoutit/Gomez:2010}. Obviously there is a
homomorphism of groups
\begin{equation}\label{Eq:Ol}
\Scr{O}_l: \Scr{P}(S/R) \longrightarrow \Picar{R},
\lr{(\objeto{[P]}{\phi}{[X]}) \longmapsto [P]}.
\end{equation}
This group also inherits the above $\G$-action. 
That is, there is a canonical left $\G$-action on this group which
is induced by $\Theta$ and given as follows:
$$
{}^x\lr{\objeto{[P]}{\phi}{[X]}}\,\,=\,\,
\lr{\objeto{[\Theta_x\tensor{R}P\tensor{R}\Theta_{x^{-1}}]}{\psi}{[X]}},
$$
where $\psi\,:=\, \Theta_x\tensor{R}\phi\tensor{R}\Theta_{x^{-1}}$
(composed with the multiplication of $S$).
The subgroup $\Scr{P}(S/R)^{\G}$ of $\G$-invariant elements of
$\Scr{P}(S/R)$
has a sightly simpler description:

\begin{lemma}\label {lema:6}
Keeping the above notations, we have 
$$
\Scr{P}(S/R)^{\G}\,\,=\,\,\left\{ \underset{}{}
\objeto{[P]}{\phi}{[X]} \in \Scr{P}(S/R)|\,\, \phi(P)\Theta_x \,=\,
\Theta_x\phi(P), \text{ for every }
x \in \G \right\},
$$ 
where $\phi(P)\Theta_x$ and $\Theta_x\phi(P)$ stand for the obvious
subsets of the $S$-bimodule $X$, e.g.
$$
\phi(P)\,\Theta_x\,\,=\,\, \left\{\underset{finite}{\sum} \phi(p_i)
\fk{u}_x^i|\,\, p_i \in P,\, \fk{u}_x^i \in \Theta_x \right\}.
$$
\end{lemma}
\begin{proof}
Is based up on two main facts. The first one is that the
$R$-bilinear map
$\phi: P \to X$ defining the given element $\objeto{[P]}{\phi}{[X]}
\in\Scr{P}(S/R)$,
is always injective, see \cite[Lemma 3.1]{ElKaoutit/Gomez:2010}. The
second fact consists on the following equivalence:
For a fixed $x \in\G$, saying that $\phi(P)\Theta_x\,=\,
\Theta_x\phi(P)$ in ${}_SX_{S}$ is equivalent to say that there is
an isomorphism
$f_x: P\cong \Theta_x\tensor{R}P\tensor{R}\Theta_{x^{-1}}$ which
completes the commutativity of the following diagram
$$
\xymatrix@C=60pt@R=25pt{ P \ar@{->}^-{\phi}[rrr] \ar@{-->}_-{f_x}[d]
& & & X \ar@{=}^-{}[d]
\\ \Theta_x\tensor{R}P\tensor{R}\Theta_{x^{-1}}
\ar@/_2pc/@{-->}_-{\psi}[rrr]
\ar@{->}^-{\Theta_x\tensor{}\phi\tensor{}\Theta_{x^{-1}}}[r]
& \Theta_x\tensor{R}X\tensor{R}\Theta_{x^{-1}} \ar@^{(->}[r] &
S\tensor{R}X\tensor{R} S \ar@{->}^-{}[r] & X. }
$$
The later means, by the definition of $\Scr{P}(S/R)$, that 
$$
\lr{\objeto{[P]}{\phi}{[X]} }\,\, =\,\,
\lr{\objeto{[\Theta_x\tensor{}P\tensor{}\Theta_{x^{-1}}]}{\psi}{[X]}}.
$$
\end{proof}

\subsection{The first exact sequence.}\label{subsec:1}

Now we set 
$$
\Scr{P}_{\Z}(S/R)\,\,=\,\,\left\{ \objeto{[P]}{\phi}{[X]}
\in\Scr{P}(S/R)|\,\, [P] \in\RPicar{\Z}{R} \underset{}{} \right\},$$
and
$$
\Scr{P}_{\Z}(S/R)^{\G}\,\, =\,\, \Scr{P}_{\Z}(S/R) \, \cap\,
\Scr{P}(S/R)^{\G}.
$$

On the other hand, we consider the group  
$\Aut{R-ring}{S}$ of all $R$-ring automorphisms of $S$ (with same set of local units). It has the following subgroup 
$$
\Aut{R-ring}{S}^{(\G)}\,\,=\,\,\left\{ f \in \Aut{R-ring}{S} |\,\,
f(\Theta_x)\,=\, \Theta_x, \text{ for all } x \in \G \underset{}{}
\right\}.
$$

There are two interesting maps which will be used in the sequel: The
first one is given by
$$
\Scr{E}: \Aut{R-ring}{S} \longrightarrow \Scr{P}(S/R), \lr{ f
\longmapsto (\objeto{[R]}{\iota_f}{[S_f]})},
$$
where $\iota_f$ is the inclusion $R \subset S_f$, and  $S_f$ is the $S$-bimodule induced from the automorphism $f$,
that is, the left $S$-action is that of ${}_SS$, while the right one
has been altered by $f$, that is, we  have a new left $S$-action 
$$s.s'\,=\, s \, f(s'), \text{ for all } s, s' \in S.$$ 
The second connects the group $\mathcal{U}(\Z)$ with
$\Aut{R-ring}{S}$,
and it is defined by 
\begin{equation}\label{Eq:F}
\Scr{F}: \mathcal{U}(\Z)\,=\,\Aut{R-R}{R} \longrightarrow
\Aut{R-ring}{S}, \lr{ \sigma \longmapsto \left[\underset{}{}
\Scr{F}(\sigma): s \mapsto \sigma^{-1}(e) \, s\, \sigma(e)\right]},\end{equation}
where $e \in Unit\{s\}$, $s \in S$.

\begin{proposition}\label{prop:6}
Keeping the previous  notations, there is
a commutative diagram whose rows are exact sequences of groups
$$
\xymatrix{ \mathcal{U}(\Z)\,=\,\Aut{R-R}{R} \ar@{=}[d]
\ar@{->}^-{\Scr{F}}[r] & \Aut{R-ring}{S} \ar@{->}^-{\Scr{E}}[r]
& \Scr{P}(S/R) \ar@{->}^-{\Scr{O}_l}[r] & \Picar{R}  \\
\mathcal{U}(\Z)\,=\, \Aut{R-R}{R} \ar@{->}[r] & \Aut{R-ring}{S}^{(\G)}
\ar@{->}[r] \ar@^{(->}[u] & \Scr{P}_{\Z}(S/R)^{\G} \ar@{->}[r]
\ar@^{(->}[u] &
\RPicar{\Z}{R}^{\G}. \ar@^{(->}[u] }
$$  
\end{proposition}
\begin{proof}
The exactness of the first row was proved in \cite[Proposition
5.1]{ElKaoutit/Gomez:2010}. The commutativity as well as the
exactness
of the second row, using routine computations, are immediately
deduced from the definition of the involved maps.
\end{proof}

The conditions under assumption, that is, the existence of $\overline{\Theta}
: \G \to \Inv{R}{S}$ a homomorphism of groups with a given
extension of rings $R \subseteq S$ with the same set of local units
$\Sf{E}$, are satisfied for the extension $R \subseteq
\Delta(\Theta)$.
Precisely, the map $\Theta$ factors thought out a homomorphism of
groups
$\overline{\Theta}': \G \to \Inv{R}{\Delta(\Theta)}, \lr{x \mapsto \Theta_x}$,
as
$R \subseteq \Delta(\Theta)$ is an extension of rings with the same
set of local units $\Sf{E}$.
Thus, applying Proposition \ref{prop:6} to this extension, we obtain
the first statement of the following corollary.

\begin{corollary}\label{coro:1}
Consider the generalized crossed product $\Delta\,:=\,\Delta(\Theta)$
of $S$ with $\G$. Then there is an exact sequence of groups
$$
\xymatrix{ 1 \ar@{->}[r] & 
\mathcal{U}(\Z) \ar@{->}[r] & \Aut{R-ring}{\Delta}^{(\G)} \ar@{->}[r]
& \Scr{P}_{\Z}(\Delta/R)^{\G} \ar@{->}[r] &
\RPicar{\Z}{R}^{\G}.  }
$$  
In particular, we have the following exact sequence of groups
$$
\xymatrix{ 1 \ar@{->}[r] & 
H_{\Theta}^1(\G, \, \mathcal{U}(\Z)) \ar@{->}^-{\Scr{S}_1}[r] &
\Scr{P}_{\Z}(\Delta/R)^{\G} \ar@{->}^-{\Scr{S}_2}[r] &
\RPicar{\Z}{R}^{\G}.  }
$$  
\end{corollary}
\begin{proof}
We only need to check the particular statement. So let $f
\in\Aut{R-ring}{\Delta}^{(\G)}$, this gives
a family of $R$-bilinear isomorphisms
$f_x:\Theta_x \to \Theta_x$, $x \in \G$, which by Lemma \ref{lema:4}
they lead to a map
$$ 
\sigma: \G \longrightarrow \mathcal{U}(\Z),\quad \lr{x \longmapsto
\widetilde{f}_x},
$$
as $\Theta_x \in \Inv{R}{S}$. An easy verification, using \eqref{Eq:t}, shows that, for
every unit $e \in\Sf{E}$, we have
\begin{equation}\label{Eq:33}
\widetilde{f}_{xy}(e)\,\,=\,\, \widetilde{f}_{x}(e)
{}^x\widetilde{f}_{y}(e).
\end{equation}
Therefore, $$\sigma_{xy}\,=\, \sigma_x \circ {}^x \sigma_y, \quad
\text{ for every } x, y \in\G.$$
That is, $\sigma$ is a normalized $1$-cocycle, i.e. $\sigma \in
Z_{\Theta}^1(\G, \mathcal{U}(\Z))$.

Conversely, given a normalized $1$-cocycle $\gamma \in Z_{\Theta}^1(\G,
\mathcal{U}(\Z))$, we consider the $R$-bilinear isomorphisms
$$g_x : \Theta_x \longrightarrow \Theta_x, \quad \lr{u \longmapsto
\gamma_x(e) u}, \text{ where } e \in Unit\{u\}, u \in\Theta_x.$$
Clearly, the direct sum $g:=\oplus_{x \in\G} g_x$ defines an
automorphism of generalized crossed product.
Hence, $g \in  \Aut{R-ring}{\Delta}^{(\G)}$. 

In conclusion we have constructed a mutually inverse maps which in
fact establishes an isomorphism of groups
\begin{equation}\label{Eq:DZU}
\Aut{R-ring}{\Delta}^{(\G)}\,\, \cong\,\, Z_{\Theta}^1(\G, \mathcal{U}(\Z)).  
\end{equation}
Now let us consider en element $\fk{u}
\in\mathcal{U}(\Z)=\Aut{R-R}{R}$, and
its image $\Scr{F}(\fk{u})$ under the map $\Scr{F}$ defined in
equation \eqref{Eq:F} using the extension
$R \subseteq \Delta$. 
Then, for every element  $a \in\Theta_x$,
we have 
\begin{eqnarray*}
\Scr{F}(\fk{u}) (a) &=& \fk{u}^{-1}(e) \,a \, \fk{u}(e),\quad e \in
Unit\{a\} \\
&\overset{\eqref{Eq:t}}{=}& \fk{u}^{-1}(e) \,{}^x\fk{u}(e) \, a.
\end{eqnarray*}
Therefore, $\fk{u}$ defines, under the isomorphism of
\eqref{Eq:DZU}, a $1$-coboundary,
$$ \G \longrightarrow \mathcal{U}(\Z),\quad \lr{x \longmapsto
{}^x\fk{u} \circ \fk{u}^{-1}}.$$
We have then constructed an isomorphism of groups 
$$ \Aut{R-ring}{\Delta}^{(\G)} / \Aut{R-R}{R} \,\, \cong \,\, H_{\Theta}^1(\G,
\, \mathcal{U}(\Z)).$$
The exactness follows now from the first statement which is a
particular case of Proposition \ref{prop:6}.
\end{proof}

\subsection{The second exact sequence.}\label{subsec:2}

For any element $[P] \in \Picar{R}$, we will denote by $[P^{-1}]$
its inverse element. Let us consider the following groups
$$
\RPicar{\Z}{R}^{(\G)}\,\,:=\,\, \left\{ \underset{}{} [P] \in
\RPicar{\Z}{R} |\,\, P\tensor{R}\Theta_x\tensor{R}P^{-1} \,\sim\,
\Theta_x,
\,\text{ for all } x \in \G \right\}.
$$
A routine computation shows that this is a subgroup of
$\RPicar{\Z}{R}$ which contains the subgroup of all $\G$-invariant
elements. That is,
there is a homomorphism of groups $\RPicar{\Z}{R}^{\G} \to
\RPicar{\Z}{R}^{(\G)}$.

Given an element $[P] \in \RPicar{\Z}{R}^{(\G)}$ and set
$\Omega_x\,:=\,P\tensor{R}\Theta_x\tensor{R}P^{-1}$, for every
$x \in\G$. Consider now the family of $R$-bilinear isomorphisms 
$$\F^{\Omega}_{x,y}: \Omega_x\tensor{R}\Omega_y \to \Omega_{xy}$$
which is defined using the factor maps $\F^{\Theta}_{x,y}$ and the
isomorphism $P\tensor{R}P^{-1} \cong R$.
This is clearly a factor map relative 
to the homomorphism of groups $\overline{\Omega}: \G \to \Inv{R}{S}$ sending 
$x \mapsto \Omega_x$. This gives us a generalized crossed product
$\Delta(\Omega)$, and in fact
establishes a homomorphism of groups 
$$ \Scr{L}: \RPicar{\Z}{R}^{(\G)} \to \Scr{C}(\Theta/R),$$
where the right hand group was defined in Section \ref{sec:3}. The
proof of the following lemma is left to the reader.
\begin{lemma}\label{lema:7}
Keep the above notations. There is a commutative diagram of groups
$$
\xymatrix{ \RPicar{\Z}{R}^{\G} \ar@^{(->}[d] \ar@{->}^-{\Scr{L}}[rr]
& & \Scr{C}_0(\Theta/R) \ar@^{(->}[d]
\\ \RPicar{\Z}{R}^{(\G)} \ar@{->}^-{\Scr{L}}[rr] & &
\Scr{C}(\Theta/R). }
$$ 
\end{lemma}

Now we can show  our second exact sequence.

\begin{proposition}\label{prop:7}
Let  $\Delta(\Theta)$ be a
generalized crossed product of $S$ with $\G$. Then there is an exact
sequence of groups
$$
\xymatrix{ \Scr{P}_{\Z}(\Delta(\Theta)/R)^{\G}
\ar@{->}^-{\Scr{S}_2}[r] & \RPicar{\Z}{R}^{\G}
\ar@{->}^-{\Scr{S}_3}[r] & \Scr{C}_0(\Theta/R), }
$$ 
where $\Scr{S}_2$ and $\Scr{S}_3$ are, respectively, the
restrictions of the map $\Scr{O}_l$ given in equation \eqref{Eq:Ol}
attached to the extension
$R \subseteq \Delta(\Theta)$, and of the map defined in Lemma
\ref{lema:7}.
\end{proposition}
\begin{proof}
Let $[P] \in \RPicar{\Z}{R}^{\G}$ such that
$\Scr{S}_3([P])\,=\,[\Delta(\Omega)]\,=\, [\Delta(\Theta)]$, where
for every $x \in\G$, $\Omega_x\,=\,
P\tensor{R}\Theta_x\tensor{R}P^{-1}$. This means that
we have an isomorphism $\Delta(\Omega) \cong \Delta(\Theta)$ of
generalized crossed products, and so a family of $R$-bilinear
isomorphisms
$\iota_x: \Theta_x \tensor{R}P \, \cong \, P\tensor{R}\Theta_x$, $x
\in \G$. We then obtain an $R$-bilinear
isomorphism
$$ \iota=\underset{x \in\,\G}{\oplus} \iota_x:\,
\Delta(\Theta)\tensor{R} P \longrightarrow
P\tensor{R}\Delta(\Theta),$$
which in fact satisfies the conditions stated in \cite[Eq. (5.1) and
(5.2), p. 161]{ElKaoutit:2010a}.
This implies that $\Delta(\Theta)\tensor{R}P$ admits a structure 
of $\Delta(\Theta)$-bimodule whose underlying left structure is
given by that of ${}_{\Delta(\Theta)}\Delta(\Theta)$. In this way,
the previous map $\iota$ becomes an isomorphism of
$(\Delta(\Theta), R)$-bimodules, and so we have
\begin{eqnarray*}
(\Delta(\Theta)\tensor{R}P )\tensor{\Delta(\Theta)}
(P^{-1}\tensor{R}\Delta(\Theta)) &\cong &
(\Delta(\Theta)\tensor{R}P )\tensor{\Delta(\Theta)}
(\Delta(\Theta)\tensor{R}P^{-1}) \\ &\cong&
\Delta(\Theta)\tensor{R}P\tensor{R}P^{-1} \\ &\cong& \Delta(\Theta),\end{eqnarray*}
an isomorphism of $\Delta(\Theta)$-bimodules. Whence the isomorphic
class of $\Delta(\Theta)\tensor{R}P$ belongs to
$\Picar{\Delta(\Theta)}$.
We have then construct an element
$\objeto{[P]}{\phi}{[\Delta(\Theta)\tensor{R}P]}$ of the group
$\Scr{P}(\Delta(\Theta)/R)$, where $\phi$ is the obvious map
$$\phi: P \longrightarrow \Delta(\Theta)\tensor{R}P,\quad \lr{ p
\longmapsto e\tensor{R}p},$$
where $e \in Unit\{p\}$. It is clear that, for every 
$x \in\G$, the equality $\phi(P)\Theta_x\,=\, \Theta_x \phi(P)$
holds true in $\Delta(\Theta)\tensor{R}P$.
Thus, $\objeto{[P]}{\phi}{[\Delta(\Theta)\tensor{R}P]} \in
\Scr{P}_{\Z}(\Delta(\Theta)/R)^{(\G)}$ and
$$\Scr{S}_2\lr{\objeto{[P]}{\phi}{[\Delta(\Theta)\tensor{R}P]}}\,=\,
[P].$$ This shows the inclusion
${\rm Ker}(\Scr{S}_3) \,\subseteq {\rm Im}(\Scr{S}_2)$.

Conversely, let $[Q] \in{\rm Im}(\Scr{S}_2)$, that is, there exists
an element $\objeto{[Q]}{\psi}{[Y]} \in
\Scr{P}_{\Z}(\Delta(\Theta)/R)^{\G}$. We need to compute the image
$\Scr{S}_3([Q])$. From the choice of $[Q]$ we know that
$[Q] \in \RPicar{\Z}{R}$ and that, for every $x \in \G$, $\psi(Q)
\Theta_x\,=\,\Theta_x \psi(Q)$ in the ${\Delta(\Theta)}$-bimodule
$Y$.
As in the proof of Lemma \ref{lema:6}, the later means that, there
are $R$-bilinear isomorphisms:
$$ f_x: \, Q\tensor{R}\Theta_x \overset{\cong}{\longrightarrow}
\Theta_x\tensor{R}Q, \text{ for every } x \in\G,$$
which convert commutative the following diagrams
$$
\xymatrix@C=40pt{ Q\tensor{R}\Theta_x\tensor{R}\Theta_y
\ar@{->}^-{Q\tensor{}\F^{\Theta}_{x,y}}[rr]
\ar@{->}_-{f_x\tensor{}\Theta_y}[dd] & & Q\tensor{R}\Theta_{xy}
\ar@{->}^-{f_{xy}}[dd] \ar@^{(->}[rd] & \\ & & & Y \\
\Theta_x\tensor{R}Q\tensor{R}\Theta_y
\ar@{->}_-{\Theta_x\tensor{}f_y}[rd] & & \Theta_{xy}\tensor{R}Q
\ar@^{(->}[ru] & \\
& \Theta_x\tensor{R}\Theta_y\tensor{R}Q
\ar@{->}_-{\F^{\Theta}_{x,y}\tensor{}Q}[ru] & & }
$$
Coming back to the image $\Scr{S}_3([Q]):=[\Delta(\Gamma)]$, we know
that
for every $x \in\G$, we have $\Gamma_x=
Q\tensor{R}\Theta_x\tensor{R}Q^{-1}\, \cong \, \Theta_x$, via the
$f_x$'s. The above commutative diagram
shows that these isomorphisms are compatible with the factor maps
$\F^{\Gamma}_{x,y}$ and $\F^{\Theta}_{x,y}$. Therefore,
$\Delta(\Theta)\,\cong\,
\Delta(\Gamma)$ as generalized crossed products, and so
$\Scr{S}_3([Q])\,=\,[\Delta(\Theta)]$ the neutral element of the
group $\Scr{C}_0(\Theta/R)$.
\end{proof}

\subsection{Comparison with a previous alternative exact sequence.}

Let us assume  here that  $\underline{\Theta}: \G \overset{\sigma}{\to} \aut{R} \to \Picar {R}$, as in  Example \ref{Example:B-Rio}.  We consider the extension $R \subset R\G$, where $S:=R\G=\Delta(\Theta)$ is the skew group ring of $R$ by $\G$.  Thus $\bara{\Theta}: \G \to \Inv{R}{S}$ factors though out $\mu$ as in diagram \eqref{Eq:factoriza}.
In what follows, we want to explain the relation between the five terms exact sequence  given in \cite[Theorem 2.8]{Beattie/Del Rio:1999}, and the resulting one by combining   Proposition \ref{prop:7} and Corollary \ref{coro:1}, that is the sequence 
$$
\xymatrix{ 1 \ar@{->}[r] & 
H_{\Theta}^1(\G, \, \mathcal{U}(\Z)) \ar@{->}^-{\Scr{S}_1}[r] & \Scr{P}_{\Z}(\Delta(\Theta)/R)^{\G}
\ar@{->}^-{\Scr{S}_2}[r] & \RPicar{\Z}{R}^{\G}
\ar@{->}^-{\Scr{S}_3}[r] &   H^2_{\Theta}(\G,\U(\Z)) \cong \Scr{C}_0(\Theta/R). }
$$ 
As we argued in Example \ref{Example:B-Rio}, the cohomology  $H^*_{\Theta}(\G,\U(\Z))$ coincides with that considered in \cite{Beattie/Del Rio:1999}. Thus, following the notation of \cite{Beattie/Del Rio:1999}, the result \cite[Theorem 2.8]{Beattie/Del Rio:1999} says that 
$$
\xymatrix{ 1 \ar@{->}[r] & 
H_{\Theta}^1(\G, \, \mathcal{U}(\Z)) \ar@{->}^-{\varphi_{\sigma}}[r] & \G\text{-}\Picar{R}
\ar@{->}^-{F_{\sigma}}[r] & \Picar{R}^{\G}
\ar@{->}^-{\Phi_{\sigma}}[r] &   H^2_{\Theta}(\G,\U(\Z)) }
$$
is an exact sequence (of pointed sets). It is not difficult to see that the restriction of $\Phi_{\sigma}$ to $\RPicar{\Z}{R}^{\G}
$ gives exactly $\Scr{S}_3$. On the other hand, for any element $\objeto{[P]}{\phi}{[X]} \in \Scr{P}_{\Z}(\Delta(\Theta)/R)^{\G}$, we will define an object $(P,\tau)$ where $\tau: \G \to {\bf Aut}_{\mathbb{Z}}(P)$  (to the group of additive automorphisms of $P$) is a $\G$-homomorphism \cite[Definition 2.3]{Beattie/Del Rio:1999}. As in the proof of  Lemma \ref{lema:6}, we know that there is a family of  $R$-bilinear isomorphisms $\Sf{j}_x:  R^x\tensor{R}P \to P\tensor{R}R^x$, $x \in \G$, so we define $\tau$ by the following composition
$$
\xymatrix{\tau_x: {}^{x^-1}P \ar@{->}^-{\cong}[r] & R^x\tensor{R}P \ar@{->}^-{\Sf{j}_x}[r] & P\tensor{R}R^x \ar@{->}^-{\cong}[r] & P^x,   }
$$
here $P^y$ denotes the $R$-bimodule constructed from $P$ by twisting its right module structure using the automorphism $y$, and the same construction is used on the left. To check that $\tau$ is a homomorphism of groups, one uses the equality 
$$
\Sf{j}_{xy} \circ (\F^{\Theta}_{x,y}\tensor{R}P)\,\,=\,\, (P\tensor{R}\F^{\Theta}_{x,y}) \circ (\Sf{j}_x\tensor{R}R^y) \circ (R^x\tensor{R}\Sf{j}_y), \quad \text{ for all } x,y \in \G.
$$
Now, if we  take another representative, that is, if we  assume that 
$$
\objeto{[P]}{\phi}{[X]}\,\,=\,\, \objeto{[P']}{\phi'}{[X']}  \in \Scr{P}_{\Z}(\Delta(\Theta)/R)^{\G},
$$
then, by the definition of the group $\Scr{P}(\Delta(\Theta)/R)$, see \cite[Section 3]{ElKaoutit/Gomez:2010}, there are bilinear isomorphisms $f: P \to P'$ and $g: X \to X'$ such that 
$$
\xymatrix{ P \ar@{->}^-{\phi}[r]  \ar@{->}_-{f}[d] & X \ar@{->}^-{g}[d] \\  P' \ar@{->}_-{\phi'}[r] & X' }
$$
commutes.  In this way we have also the following commutative diagram 
$$
\xymatrix{ R^x\tensor{R}P \ar@{->}^-{\Sf{j_x}}[r]  \ar@{->}_-{R^x\tensor{}f}[d] & P\tensor{R}R^x \ar@{->}^-{f\tensor{}R^x}[d] \\  R^x\tensor{R}P' \ar@{->}_-{\Sf{j}'_x}[r] &  P'\tensor{R}R^x,}
$$ which says that $\tau'_x \circ f\,=\, f \circ \tau_x$. This implies that $[P,\tau] \,=\, [P',\tau']$ in the group $\G\text{-}\Picar{R}$. We then have constructed a monomorphism of groups $\Scr{P}_{\Z}(\Delta(\Theta)/R)^{\G} \to \G\text{-}\Picar{R}$ such that 
the restriction of the map $F_{\sigma}$ to  $\Scr{P}_{\Z}(\Delta(\Theta)/R)^{\G}$ is exactly $\Scr{S}_2$.  We leave to the reader to check that we furthermore obtain a commutative diagram
$$
\xymatrix{ 1 \ar@{->}[r] & 
H_{\Theta}^1(\G, \, \mathcal{U}(\Z)) \ar@{=}[d]  \ar@{->}^-{\Scr{S}_1}[r] & \Scr{P}_{\Z}(\Delta(\Theta)/R)^{\G} \ar@{^{(}->}[d]
\ar@{->}^-{\Scr{S}_2}[r] & \RPicar{\Z}{R}^{\G} \ar@{^{(}->}[d]
\ar@{->}^-{\Scr{S}_3}[r] &   H^2_{\Theta}(\G,\U(\Z)) \ar@{=}[d] \\
1 \ar@{->}[r] & 
H_{\Theta}^1(\G, \, \mathcal{U}(\Z))   \ar@{->}^-{\varphi_{\sigma}}[r] & \G\text{-}\Picar{R}
\ar@{->}^-{F_{\sigma}}[r] & \Picar{R}^{\G}
\ar@{->}^-{\Phi_{\sigma}}[r] &   H^2_{\Theta}(\G,\U(\Z)).  
}
$$

\subsection{The third exact sequence.}\label{subsec:3}
We define the group $\Scr{B}(\Theta/R)$ as the quotient group
\begin{equation}\label{Eq:B}
\xymatrix{ \RPicar{\Z}{R}^{(\G)} \ar@{->}^-{\Scr{L}}[r] &
\Scr{C}(\Theta/R) \ar@{->}[r] & \Scr{B}(\Theta/R) \ar@{->}[r] & 1,}
\end{equation}
where $\Scr{L}$ is the map defined in Lemma \ref{lema:7}.

The third exact sequence of groups is stated in the following

\begin{proposition}\label{prop:8}
Let $\Delta(\Theta)$ be a
generalized crossed product of $S$ with $\G$. Then there is an exact
sequence of groups
$$
\xymatrix{ \RPicar{\Z}{R}^{\G} \ar@{->}^-{\Scr{S}_3}[r] &
\Scr{C}_0(\Theta/R) \ar@{->}^-{\Scr{S}_4}[r] & \Scr{B}(\Theta/R), }
$$
where $\Scr{S}_3$ and $\Scr{S}_4$ are constructed from the sequence
\eqref{Eq:B} and the
commutativity of the diagram stated in  Lemma \ref{lema:7}.
\end{proposition}
\begin{proof}
The inclusion ${\rm Im}(\Scr{S}_3) \subseteq {\rm Ker}(\Scr{S}_4)$
is by construction trivial.
Now, let $[\Delta(\Gamma)] \in \Scr{C}_0(\Theta/R)$ such that
$\Scr{S}_4([\Delta(\Gamma)])\,=\, 1$. By the definition of the group
$\Scr{B}(\Theta/R)$, there exists $[P] \in \RPicar{\Z}{R}^{(\G)}$
such that $\Scr{L}([P])\,=\, [\Delta(\Gamma)]$. Therefore, we obtain
a chain of $R$-bilinear isomorphisms:
$$ \Theta_x\,\,\cong\,\, \Gamma_x\,\,\cong\,\,
P\tensor{R}\Theta_x\tensor{R}P^{-1}, \text{ for all } x \in \G,$$
which means that $[P] \in \RPicar{\Z}{R}^{\G}$, and so
$[\Delta(\Gamma)]=\Scr{S}_3([P])$ which completes the proof.
\end{proof}

\begin{remark}
In the case of a unital commutative Galois extension $\Bbbk \subseteq
R$
with a finite Galois group, see for instance
\cite[p. 396]{Auslander/Goldman:1960}, 
and taking the extension $R \subseteq S$ and $\Theta$ as in Remark
\ref{rem:com}.
The homomorphism $\Scr{S}_4$ coincides, up to the isomorphism of
Proposition \ref{prop:5}, with
the homomorphism of groups defined in \cite[Theorem A 12, p.
406]{Auslander/Goldman:1960} and where the group
$\Scr{B}(\Theta/R)$ coincides with the Brauer group of the
$\Bbbk$-Azumaya algebras split by $R$.
\end{remark}

\subsection{The fourth exact sequence.}\label{subsec:4}

Consider the following subgroup of the Picard group $\Picar{R}$:
$$
\RPicar{0}{R}\,\,=\,\, \left\{ \underset{}{} [P] \in \Picar{R}|\,\,
P\,\sim\, R, \text{ as bimodules} \right\}.
$$
As we have seen in Proposition \ref{prop:1}, this is an abelian
group.
Clearly it inherits the $\G$-module structure of $\Picar{R}$: The
action is given by
$$
{}^x[P]\,=\, [\Theta_x\tensor{R}P\tensor{R}\Theta_{x^{-1}}], \text{
for every } x \in\G.
$$

\begin{lemma}\label{lema:8}
Keep the above notations. Then the map
$$
\xymatrix@R=0pt{ \Scr{C}(\Theta/R) \ar@{->}^-{\zeta}[rr] & &
Z^{1}(\G,\RPicar{0}{R})
\\ [\Delta(\Gamma)] \ar@{|->}[rr] & & \left[\underset{}{} x
\longmapsto [\Gamma_x][\Theta_{x^{-1}}] \right] }
$$ 
defines a homomorphism of groups. Furthermore, there is an exact
sequence of groups
$$
\xymatrix{ 1\ar@{->}[r] & \Scr{C}_0(\Theta/R) \ar@{->}^-{}[r] &
\Scr{C}(\Theta/R) \ar@{->}^-{\zeta}[r] & Z^{1}(\G,\RPicar{0}{R}). }
$$
\end{lemma}
\begin{proof}
The first claim is immediate from the definitions, as well as the
inclusion $\Scr{C}_0(\Theta/R) \subseteq {\rm Ker}(\zeta)$.
Conversely, given an element $[\Delta(\Gamma)] \in
\Scr{C}(\Theta/R)$ such that $\zeta([\Delta(\Gamma)])\,=\,1$.
This implies that for every $x \in\G$, there is an isomorphism of
$R$-bimodule
$\Gamma_x\tensor{R}\Theta_{x^{-1}} \,\cong\, R$. Hence, $\Gamma_x
\,\cong\, \Theta_x$, for every $x \in\G$, which means
that $[\Delta(\Gamma)] \in\Scr{C}_0(\Theta/R)$.
\end{proof}

We define another abelian group, which in the commutative Galois
case \cite{Kanzaki:1968}, coincides with the first
cohomology group of $\G$ with coefficients in $\RPicar{0}{R}$. It is
defined by the following commutative diagram
$$
\xymatrix{ \RPicar{\Z}{R}^{(\G)} \ar@{-->}[rr]
\ar@{->}_-{\Scr{L}}[rd] & & Z^{1}(\G,\RPicar{0}{R}) \ar@{->}^-{}[r]
& \bara{H}^{1}(\G,\RPicar{0}{R}) \ar@{->}[r] & 1 \\ &
\Scr{C}(\Theta/R) \ar@{->}_-{\zeta}[ru] & & & }
$$
whose  row is an exact sequence. 
Our fourth exact sequence is given by the following.

\begin{proposition}\label{prop:9}
Let $\Delta(\Theta)$ be a
generalized crossed product of $S$ with $\G$. Then there is an exact
sequence of groups
$$
\xymatrix{  
\Scr{C}_0(\Theta/R) \ar@{->}^-{\Scr{S}_4}[r] & \Scr{B}(\Theta/R)
\ar@{->}^-{\Scr{S}_5}[r] & \bara{H}^{1}(\G,\RPicar{0}{R}). }
$$
\end{proposition}
\begin{proof}
First we need to define $\Scr{S}_5$. 
We construct it as the map which completes the commutativity of the
diagram
$$
\xymatrix{ & \Scr{C}_0(\Theta/R) \ar@{^{(}->}^-{}[d]
\ar@{->}^-{\Scr{S}_4}[rd] & & \\
\RPicar{\Z}{R}^{(\G)} \ar@{->}^-{\Scr{L}}[r] & \Scr{C}(\Theta/R)
\ar@{->}^-{\Scr{L}^c}[r] \ar@{->}_-{\zeta}[d]
& \Scr{B}(\Theta/R) \ar@{-->}^-{\Scr{S}_5}[d] \ar@{->}[r] & 1 \\ 
& Z^{1}(\G,\RPicar{0}{R}) \ar@{->}[r] &
\bara{H}^{1}(\G,\RPicar{0}{R}) \ar@{->}[r] & 1, }
$$
whose first row is exact. 

The fact that ${\rm Im}(\Scr{S}_4) \subseteq {\rm Ker}(\Scr{S}_5) $
is easily deduced from the previous
diagram and Lemma \ref{lema:8}. 
Conversely, let $\Xi\,=\,\Scr{L}^c([\Delta(\Gamma)])$, for some
$[\Delta(\Gamma)] \in \Scr{C}(\Theta/R)$,
be an element in $\Scr{B}(\Theta/R)$ (here $f^c$ denotes the
cokernel map of $f$),
such that $\Scr{S}_5(\Xi)\,=\,\Scr{S}_5\circ
\Scr{L}^c([\Delta(\Gamma)])\,=\, 1$. Then $(\zeta\Scr{L})^c \circ
\zeta([\Delta(\Gamma)])\,=\,1$,
which implies that $\zeta([\Delta(\Gamma)]) \in
\zeta\lr{\Scr{L}(\RPicar{\Z}{R}^{(\G)})}$. Therefore, there exists
an element
$[P] \in \RPicar{\Z}{R}^{(\G)}$ such that
$[\Delta(\Gamma)]\Scr{L}([P])^{-1} \in {\rm Ker}(\zeta)\,=\,
\Scr{C}_0(\Theta/R)$ by Lemma
\ref{lema:8}. Now, we have 
\begin{eqnarray*}
\Scr{S}_4\lr{[\Delta(\Gamma)]\Scr{L}([P])^{-1}} &=&
\Scr{S}_4\lr{[\Delta(\Gamma)]\Scr{L}([P^{-1}])}
\\ &=& \Scr{L}^c([\Delta(\Gamma)])\Scr{L}^c\Scr{L}([P^{-1}]) \\ &=&
\Scr{L}^c([\Delta(\Gamma)]) \,\,=\,\, \Xi,
\end{eqnarray*}
and this shows that $\Xi \in {\rm Im}(\Scr{S}_4)$.
\end{proof}

\subsection{The fifth exact sequence and the main
Theorem.}\label{subsec:5}

Keep from Subsection \ref{subsec:4} the definition of the group
$\RPicar{0}{R}$ with its structure of $\G$-module.
Before stating the fifth sequence, we will need first to give a
homomorphism of group from
the group $\bara{H}^1(\G,\RPicar{0}{R})$ to the third cohomology
group $H_{\Theta}^3(\G, \mathcal{U}(\Z))$.

Given a normalized $1$-cocycle $g \in Z^1(\G,\RPicar{0}{R})$ and put
$g_x\,=\, [\nabla_x]$. Then, for every pair of elements $x, y
\in\G$, one
can easily shows that 
\begin{equation}\label{Eq:nabla}
{}^xg_y\,[\Theta_x]\,\,=\,\, [\Theta_x] \, g_y.
\end{equation}
Now, for every $x \in\G$, we set 
$$ [U_x] \,\,:=\,\, g_x [\Theta_x], \text{ in } \Picar{R}.$$
By the cocycle condition, we obtain 
\begin{eqnarray*}
[U_{xy}] &=& g_{xy} [\Theta_{xy}] \\ &=& g_x
{}^xg_y[\Theta_x][\Theta_y] \\ &\overset{\eqref{Eq:nabla}}{=}&
g_x [\Theta_x]\,  g_y[\Theta_y] \\ &=& [U_x][U_y].
\end{eqnarray*}
This means that there are $R$-bilinear isomorphisms 
\begin{equation}\label{Eq:Fg}
\F^g_{x,y} : U_x\tensor{R}U_y \longrightarrow U_{xy}, 
\end{equation}
with $\F^g_{1,x}\,=\, id\,=\, \F^g_{x,1}$ for every $x \in\G$. By
Proposition \ref{prop:3}, we have a $3$-cocycle
in $Z_{\Theta}^{3}(\G, \mathcal{U}(\Z))$ attached to these maps
$\F^{g}_{-,-}$, which we denote by $\alpha^{g}_{-,-,-}$.
If there is another class $[V_x] \in \Picar{R}$ such that
$[V_x]\,=\, g_x [\Theta_x]$, for every $x \in\G$.
Then the families of invertible $R$-bimodules $\{V_x\}_{x\, \in\G}$
and $\{U_x\}_{x\, \in\G}$
satisfy the conditions of Proposition \ref{prop:31}. Therefore, the
associated $3$-cocycles are cohomologous.
This means that the correspondence $g \mapsto [\alpha^{g}_{-,-,-}]
\in H_{\Theta}^3(\G, \mathcal{U}(\Z))$ is a well defined map.
Henceforth,  we have a well defined homomorphism of groups 
\begin{equation}\label{Eq:alphag}
\xymatrix@R=0pt{ Z^{1}(\G,\RPicar{0}{R})
\ar@{->}^-{\Scr{S}_{13}}[rr] & & H_{\Theta}^3(\G, \mathcal{U}(\Z)) \\ g
\ar@{|->}[rr] & & [\alpha^{g}_{-,-,-}]. } 
\end{equation}

The proof of the fact that $\Scr{S}_{13}$ is a multiplicative map
uses Lemma \ref{lema:1}(ii) and the twisted
natural transformation of Proposition \ref{prop:1}. Completes and
detailed steps of this proof are omitted.

\begin{proposition}\label{prop:10}
Let $R \subseteq S$ be an extension of rings with the same set of
local units, and $\Delta(\Theta)$ a
generalized crossed product of $S$ with $\G$. Then the homomorphism
of equation \eqref{Eq:alphag} satisfies
$\Scr{S}_{13} \circ \zeta \,\,=\,\,[1]$, where $\zeta$ is the
homomorphism of Lemma \ref{lema:8}.
Furthermore, there is a commutative diagram of groups
$$
\xymatrix@R=35pt@C=50pt{ \RPicar{\Z}{R}^{(\G)}
\ar@{->}^-{(\zeta\circ \Scr{L})}[r] &
Z^{1}(\G,\RPicar{0}{R}) \ar@{->}_-{\Scr{S}_{13}}[d] \ar@{->}[r] &
\bara{H}^1(\G,\RPicar{0}{R})
\ar@{-->}^-{\Scr{S}_6}[ld] \ar@{->}[r] & 1
\\    &  H_{\Theta}^3(\G, \mathcal{U}(\Z))  & &   }
$$
with exact row. 
\end{proposition}
\begin{proof}
We need to check that $\Scr{S}_{13} \circ \zeta([\Delta(\Gamma)])
\,\,=\,\,[1]$, for any element
$[\Delta(\Gamma)] \in \Scr{C}(\Theta/R)$. Set
$f\,=\,\zeta([\Delta(\Gamma)])$, so we have
$f_x[\Theta_x]\,=\,[\Gamma_x]$, for every $x \in \G$. Thus, the
corresponding maps
$\F^f_{-.-}$ of equation \eqref{Eq:Fg}, are exactly given by the
factors maps
$\F^{\Gamma}_{-,-}$ of $\Delta(\Gamma)$ relative to $\Gamma$. They
define an associative multiplication
on $\Delta(\Gamma)$, and so induce a trivial $3$-cocycle, which
means that $\Scr{S}_{13} (f)\,=\,[1]$.
The last statement is now clear.
\end{proof}

At this level we are ending up with the following commutative
diagram
$$
\xymatrix@R=20pt@C=40pt{ & \RPicar{\Z}{R}^{\G} \ar@{^{(}->}^-{}[d]
\ar@{->}^-{\Scr{S}_3}[r] & \Scr{C}_0(\Theta/R) \ar@{^{(}->}^-{}[d]
\ar@{->}^-{\Scr{S}_4}[rd] & & \\
1 \ar@{->}[r] & \RPicar{\Z}{R}^{(\G)} \ar@{=}[d]
\ar@{->}^-{\Scr{L}}[r] & \Scr{C}(\Theta/R)
\ar@{->}_{\zeta}[d] \ar@{->}^-{\Scr{L}^c}[r] & \Scr{B}(\Theta/R)
\ar@{->}[r] \ar@{->}^-{\Scr{S}_5}[d] & 1
\\ & \RPicar{\Z}{R}^{(\G)}  \ar@{->}^-{(\zeta\circ \Scr{L})}[r] & 
Z^{1}(\G,\RPicar{0}{R}) \ar@{->}^-{\Scr{S}_{13}}[d]
\ar@{->}^-{(\zeta\circ \Scr{L})^c}[r] & \bara{H}^1(\G,\RPicar{0}{R})
\ar@{->}^-{\Scr{S}_6}[ld] \ar@{->}[r] & 1
\\  &  &  H_{\Theta}^3(\G, \mathcal{U}(\Z))  &  &   }
$$

\begin{proposition}\label{prop:11}
Let  $\Delta(\Theta)$ be a
generalized crossed product of $S$ with $\G$. Then there is an exact
sequence of groups
$$
\xymatrix{ \Scr{B}(\Theta/R) \ar@{->}^-{\Scr{S}_5}[r] &
\bara{H}^1(\G,\RPicar{0}{R}) \ar@{->}^-{\Scr{S}_6}[r] &
H_{\Theta}^3(\G, \mathcal{U}(\Z)). }
$$
\end{proposition}
\begin{proof}
It is clear from the above diagram that ${\rm Im}(\Scr{S}_5)
\subseteq {\rm Ker}(\Scr{S}_6)$
since by Proposition \ref{prop:10} $\Scr{S}_{13} \circ \zeta\,=\,
[1]$. Conversely, let us consider
a class $[h] \in \bara{H}^1(\G,\RPicar{0}{R})$ such that
$\Scr{S}_6([h])\,=\, 1$, for some element
$h \in Z^{1}(\G,\RPicar{0}{R})$. By the same diagram, 
we also have that the class $\Scr{S}_{13}(h) \,=\, [1] \in H_{\Theta}^3(\G,
\mathcal{U}(\Z))$.
This says that $\Scr{S}_{13}(h)\,=\,[\beta]$, for some $ \beta \in
B_{\Theta}^2(\G, \mathcal{U}(\Z))$,
and so there exists $\sigma_{-,-} : \G\times\G \to \mathcal{U}(\Z)$
such that
$$ \beta_{x,y,z}\,\,=\,\, \sigma_{x,yz}\, {}^x\sigma_{y,z}\,
\sigma_{x,y}^{-1}\, \sigma_{xy,z}^{-1},$$
for every $x, y,, z \in\G$. Now, for each $x \in\G$, we put
$h_x[\Theta_x]\,=\, [\Omega_x]$, and consider
$\F^{h}_{-,-}$ the associated family of maps as in equation
\eqref{Eq:Fg},
$\F^{h}_{x,y}: \Omega_x\tensor{R}\Omega_y \to \Omega_{xy}$, where
$x, y \in\G$. By definition
the $\beta_{-,-,-}$'s satisfy 
$$ \beta_{x,y,z} \circ \F^{h}_{x,yz} \circ
\lr{\F^{h}_{x,y}\tensor{R}\Omega_z}\,\,=\,\, \F^{h}_{xy,z}
\circ \lr{\Omega_x\tensor{R}\F^{h}_{y,z}},$$
for every $x, y,, z \in\G$. In this way, there are factor maps
$\F^{\Omega}_{-,-}$ relative to $\Omega$, defined by
$$
\xymatrix@R=0pt{ \F^{\Omega}_{x,y}: \Omega_x \tensor{R}\Omega_y
\ar@{->}[rr] & & \Omega_{xy} \\ \omega_x\tensor{R}\omega_y
\ar@{|->}[rr] & &
\sigma_{x,y}(e) \F^{h}_{x,y}(\omega_x\tensor{R}\omega_y), }
$$
where $e \in Unit\{\omega_x,\,\omega_y\}$. A routine computation
shows that $\Delta(\Omega)$ is actually a generalized crossed
product
of $S$ with $\G$ with factor maps $\F^{\Omega}_{-,-}$ relative to
$\Omega$.

On the other hand, since for each $x \in \G$, $h_x \in
\RPicar{0}{R}$, we have $\Omega_x \, \sim \, \Theta_x$. Thus
$[\Delta(\Omega)] \in
\Scr{C}(\Theta/R)$  with $\zeta([\Delta(\Omega)])\,=\, h$. Whence 
$$(\zeta\Scr{L})^c(h)\,=\, [h]\,=\, (\zeta\Scr{L})^c\circ \zeta
([\Delta(\Omega)])\,=\, \Scr{S}_5(\Scr{L}^c([\Delta(\Omega)])),$$
which implies that $[h]\in {\rm Im}(\Scr{S}_5)$, and this completes
the proof.
\end{proof}

The following diagram summarizes the information we have show so
far.

$$
\xymatrix@R=20pt@C=30pt{ H_{\Theta}^1(\G, \mathcal{U}(\Z))
\ar@{->}^-{\Scr{S}_1}[r] & \Scr{P}_{\Z}(\Delta/R)
\ar@{->}^-{\Scr{S}_2}[r] &
\RPicar{\Z}{R}^{\G} \ar@{^{(}->}^-{}[d] \ar@{->}^-{\Scr{S}_3}[r] &
H_{\Theta}^2(\G, \mathcal{U}(\Z)) \ar@{^{(}->}^-{}[d]
\ar@{->}^-{\Scr{S}_4}[rd] & \\
& & \RPicar{\Z}{R}^{(\G)} \ar@{=}[d] \ar@{->}^-{\Scr{L}}[r] &
\Scr{C}(\Theta/R)
\ar@{->}_{\zeta}[d] \ar@{->}^-{\Scr{L}^c}[r] & \Scr{B}(\Theta/R)
\ar@{->}^-{\Scr{S}_5}[d]
\\ & & \RPicar{\Z}{R}^{(\G)}  \ar@{->}^-{(\zeta\circ \Scr{L})}[r] &Z^{1}(\G,\RPicar{0}{R}) \ar@{->}^-{\Scr{S}_{13}}[d]
\ar@{->}^-{(\zeta\circ \Scr{L})^c}[r] & \bara{H}^1(\G,\RPicar{0}{R})
\ar@{->}^-{\Scr{S}_6}[ld]
\\ &  &  &  H_{\Theta}^3(\G, \mathcal{U}(\Z))  &     }
$$

We are now in  position to announce our main result.

\begin{theorem}\label{them}
Let $R \subseteq S$ be an extension of rings with the same set of
local units, and $\Delta(\Theta)$ a
generalized crossed product of $S$ with a group $\G$. Then there is
an exact sequence of groups
\begin{footnotesize}
$$
\xymatrix{1 \ar@{->}[r] & 
H_{\Theta}^1(\G, \, \mathcal{U}(\Z)) \ar@{->}^-{\Scr{S}_1}[r] &
\Scr{P}_{\Z}(\Delta/R)^{\G} \ar@{->}^-{\Scr{S}_2}[r] &
\RPicar{\Z}{R}^{\G} \ar@{->}^-{\Scr{S}_3}[r] &
H_{\Theta}^{2}(\G,\mathcal{U}(\Z))
   \ar@{->}^-{\Scr{S}_4}[r]  & \Scr{B}(\Theta/R)   
\\ & & \ar@{->}^-{\Scr{S}_5}[r] & \bara{H}^1(\G,\RPicar{0}{R})
\ar@{->}^-{\Scr{S}_6}[r] &
H_{\Theta}^3(\G, \mathcal{U}(\Z)). &   }
$$ 
\end{footnotesize}

\end{theorem}
\begin{proof}
This is a direct consequence of Corollary \ref{coro:1} and
Propositions
\ref{prop:7}, \ref{prop:8},  \ref{prop:9}, and \ref{prop:11}.
\end{proof}

\end{document}